\documentclass[11pt,reqno]{amsart}

\usepackage{amsmath,amsthm,amssymb,comment,fullpage}
\usepackage{braket}
\usepackage{mathtools}
\usepackage{mathtools}
\usepackage{amsthm}
\usepackage{amsmath}
\usepackage{amsfonts}
\usepackage{amssymb}
\usepackage{amsrefs}
\usepackage{graphicx}
\usepackage[flushleft]{threeparttable}
\usepackage{float} 
\usepackage{hyperref}

\newcommand{\sps}{\vspace{3pt}}   
\newcommand{\spm}{\vspace{6pt}}

\usepackage{setspace}

 \newcommand{\GL}{\operatorname{GL}}
 
 \newcommand{\Div}{\operatorname{Div}}
  
 \newcommand{\Aut}{\operatorname{Aut}}
 \newcommand{\NS}{\operatorname{NS}}
 \newcommand{\rank}{\operatorname{rank}}
 \newcommand{\Hom}{\operatorname{Hom}}
 \newcommand{\End}{\operatorname{End}}
 \newcommand{\charec}{\operatorname{char}}
 
 \newcommand{\disc}{\operatorname{disc}}
    
     \newcommand{\cont}{\operatorname{cont}}

     \newcommand{\adj}{\operatorname{adj}}
     \newcommand{\cl}{\operatorname{cl}}

      \newcommand{\lcm}{\operatorname{lcm}}

\newcommand{\Mod}[1]{\ (\mathrm{mod}\ #1)}

\newcommand{\Q}{\mathbb{Q}}
\newcommand{\Z}{\mathbb{Z}}

\theoremstyle{plain}
\newtheorem{thm}{Theorem}[section]
\newtheorem{lem}[thm]{Lemma}
\newtheorem{cor}[thm]{Corollary}
\newtheorem{prop}[thm]{Proposition}

\theoremstyle{definition}

\newtheorem{ex}[thm]{Example}

\theoremstyle{remark}
\newtheorem{remark}[thm]{Remark}

\numberwithin{equation}{section}

\DeclareTextFontCommand{\emph}{\em} 

\begin{document}

 \title[Automorphism Group of a Genus 2 Curve]{The Refined Humbert Invariant for an Automorphism Group of a Genus 2 Curve}


\author{Harun K{\i}r}
\address{Department of Mathematics and Statistics, Jeffery Hall, 99 University Avenue,
Queen’s University, Kingston, Ontario, K7L 3N6, Canada}
\curraddr{}
\email{19hk16@queensu.ca}
\thanks{Key words and phrases. Curves of genus 2, automorphism groups, ternary quadratic forms, elliptic curves with complex multiplication, product surfaces, Jacobians, Humbert invariant, Humbert surfaces, elliptic subcovers.}

\date{}

\begin{abstract}
The purpose of this paper is to list the refined Humbert invariants for a given automorphism group of a curve $C/K$ of genus 2 over an algebraically closed field $K$ with $\charec(K) = 0$.  This invariant is an algebraic generalization of the (usual) \textit{Humbert invariant}. It is a  positive definite quadratic form associated to the curve $C$, and it encodes many geometric properties of the curve. The paper has a special interest in the cases where $\Aut(C)\simeq D_4$ or $D_6$. In these cases, several applications of the main results are discussed including the curves with elliptic subcovers of a given degree.
\end{abstract}

\maketitle


\section{Introduction}

The curves of genus 2 with a group of automorphisms  has been considered by many authors, e.g. see the references in \S1.5 of \cite{accola}, \cite{Igusa}, and \cite{cardona_quer}. The classification of these curves for a given automorphism group is achieved in different ways; cf.\ \cite{accola}, \cite{cardona_quer}, \cite{shaska2004elliptic}. In the present article, we also study classifying these curves for a given automorphism group, controlling the associated quadratic forms. In doing so, we obtain interesting applications to the curves of genus 2; cf.\ \S\ref{generalizedHumbert} and \ref{s: elliptic_subcovers} below. 

Let $C/K$ be a curve of genus 2 over an algebraically closed field $K$ with characteristic 0. Kani\cite{kani1994elliptic} introduced a positive (definite) quadratic form $q_C$ attached to the curve $C;$ (cf.\ \cite{kani2014jacobians} or \S \ref{refinedhumbert} below.) This form is called the \textit{refined Humbert invariant of} $C.$ 
 This invariant is a useful ingredient since it translates geometric problems into arithmetic problems. An elegant illustration of making use of the refined Humbert invariant can be found in Kani\cite{kani2014jacobians},\cite{MJ}, and in the author's Ph.D. thesis.

 Kani showed that the automorphism group $\Aut(C)$ of a genus 2 curve $C$ can be determined from its associated refined Humbert invariant $q_C;$ cf.\ Theorem 25 of \cite{cas} (or Table \ref{Tab:autC} below).
 Moreover,  in certain cases, he gave a list of the refined Humbert invariants having the properties in his result; cf.\ Proposition 28 of \cite{cas}. 
 By extending the idea of the argument in the proof of his result, we give a complete list of the possible  ternary quadratic forms.  

\begin{thm}   
\label{firstmainresultPART2}
Let $q_C$ be a  ternary quadratic form, so its Jacobian $J_C$ is isogenous to a product $E\times E$ of an elliptic curve $E/K$ with complex multiplication.
    Then $|\Aut(C)| > 2$ if and only if the refined Humbert invariant $q_C$ is equivalent to one of the ternary quadratic  forms listed in \emph{Propositions \ref{autgl3}, \ref{classificationD6}, \ref{classificationD4} and \ref{classificationC2_C2}} below.
    
    In addition, if $q$ is one of the forms listed in \emph{Propositions  \ref{autgl3}, \ref{classificationD6}, \ref{classificationD4} and \ref{classificationC2_C2}(i)} below, then there is a curve $C/K$  of genus $2$  
such that $q_C$ is equivalent to $q.$ 
\end{thm}

We have actually finer results than what was stated in Theorem \ref{firstmainresultPART2}. More precisely, Propositions  \ref{autgl3}, \ref{classificationD6}, \ref{classificationD4} and \ref{classificationC2_C2} below explicitly show which ternary (quadratic) forms belong to a specific automorphism group; cf.\ Table \ref{Tab:k(C)} below.  Note that we excluded the forms listed in Proposition \ref{classificationC2_C2}(ii) in the last assertion of Theorem \ref{firstmainresultPART2}. Those forms are primitive, and it is hard for those forms to decide whether they are equivalent to a form $q_C$ or not. By using the complete classification of the primitive ternary forms $q_C$ in \cite{refhum}, even though we will discuss whether $q\sim q_C,$ for some curve $C$ for all primitive forms $q$ listed in Propositions \ref{classificationD6}, \ref{classificationD4} and \ref{classificationC2_C2}(i), we are not able to do the same thing for the primitive forms listed in Proposition \ref{classificationC2_C2}(ii). The main reason is that it seems that our method  in the proof of Proposition \ref{all_primitiveforms_humbert} below is not applicable  for this case.

One can apply the list provided by Theorem \ref{firstmainresultPART2} to the interesting topics discussed in Cardona and Quer\cite{cardona_quer} and Kani\cite{ESCII}, \cite{SubcoversofCurves}. More precisely, they characterized genus 2 curves whose automorphism group contains (up to isomorphism) the dihedral group $D_n$ of order $2n,$ for $n=6 \textrm{ and }4$ in different ways. With respect to our approach among them, Kani proved in Theorem 4(a) of \cite{ESCII} that:
\begin{equation}
\label{eq: statement_thm_4a_D6}
     q_C\text{ primitively represents the binary form } 4 x^2+4xy+ 4 y^2\ \Leftrightarrow \ D_6\leqslant \Aut(C),
\end{equation}
\begin{equation}
\label{eq: statement_thm_4a_D4}
     q_C \text{ primitively represents the binary form } 4 x^2+4 y^2 \ \Leftrightarrow \ D_4\leqslant \Aut(C).
\end{equation}
In the special case of Theorem \ref{firstmainresultPART2} that $|\Aut(C)|\geq 12$,  we have a simple explicit list for the characterization of (\ref{eq: statement_thm_4a_D6}):
 
 \begin{thm} 
 \label{[4,4,4]}
Let $C/K$ be a curve of genus $2.$ Then  $\Aut(C)$ contains a subgroup isomorphic to the dihedral group $D_6$ if and only if the refined Humbert invariant $q_C$ is equivalent to $4x^2+4xy+4y^2$ or  
     $4x^2+4y^2+cz^2+4yz+4xz+4xy\text{ or } 4x^2+4y^2+cz^2-4xy,$ where $c\equiv0,1\Mod{4},$ and $c>1.$
     
     In addition, if $q$ is one of the forms listed in the previous assertion, then there is a curve $C/K$ of genus $2$  such that $q_C$ is equivalent to $q.$ 
\end{thm}

 The first assertion of this result has been already proved for odd numbers $c;$ cf.\ Proposition 28 of \cite{cas} and (\ref{eq: statement_thm_4a_D6}).   There is a similar explicit list (but more complicated) in the case that $D_4;$ cf.\ Theorem \ref{[4,0,4]} below.

We will see in Section \ref{generalizedHumbert} that Theorem \ref{[4,4,4]} (and Theorem \ref{[4,0,4]} below) lead to some interesting results. To state more precisely, we recall that Kani\cite{MJ} introduced the concept of a \textit{generalized Humbert scheme} $H(q)$ which is associated with a given quadratic form $q$. This set is defined by using the refined Humbert invariant, and it is a generalization of a \textit{Humbert surface $H_n$ of invariant} $n$; cf.\ \S3 of \cite{MJ} or \S\ref{generalizedHumbert} below (also, see van der Geer\cite{van2012hilbert}, Ch. IX for a contemporary treatment of Humbert surfaces).

It is interesting  to understand the intersections of the generalized Humbert schemes and/or Humbert surfaces; cf.\ \cite{SubcoversofCurves}. 
For example, McMullen\cite{Mcmullen} posed a question regarding the description of the intersection of Humbert surfaces $H_N \cap H_M$.  He solved  the problem for the intersection $H_N \cap H_1$. 
In particular, Kani\cite{SubcoversofCurves}, Proposition 5.2(d), gave a formula/algorithm for the intersection $H_n\cap H_m$ of two Humbert surfaces; cf.\ Example \ref{ex: intersections_3_humbert_surfaces} below. This formula involves the $H(q)$'s for certain binary  forms $q$'s.  The results above can be used to give a formula for the intersection $H(q)\cap H_m$  for a binary  form $q$ and integer $m$ in certain cases; cf.\ Corollaries \ref{intersectioncorollary} and \ref{intersectioncorollary2} below. This formula involves the $H(q)$'s for certain ternary forms $q$'s. Hence, our formulas in those corollaries can be seen as a contribution to McMullen's question.

Recently, by Corollary 4 of \cite{Kir} we have that $H_N \cap H_M \neq \varnothing$ if $H_N$ and $H_M$ are not empty. This leads to the question of whether a similar result is true for the generalized Humbert schemes $H(q),$ where $q$ is binary. It turns out that the answer is negative.

\begin{cor}
\label{emptyH(q1)H(q2)}
Let $q_1=4x^2+4y^2$ and $q_2=9x^2+6xy+13y^2.$ Then the $H(q_i)$ are irreducible curves, for $i=1,2,$ but $H(q_1)\cap H(q_2) = \varnothing$.
\end{cor}

 The list provided in Theorem \ref{firstmainresultPART2}  has another application concerning the integer $k(q_C) = |\Aut^+(q_C)|/|\Aut(C)|$ as follows:  If $|\Aut(C)|>2,$ where $q_C$ is a ternary form, then $$k(q_C)\leq 2, \text{ with the exception of one class of forms for which }k(q_C)=4,$$ cf.\ Proposition \ref{theoremk(q)} below. Rather than an upper bound, we have actually explicit information for the formula $k(q_C)$ in virtually all cases; cf.\ Table \ref{Tab:k(C)} below and the proof of Proposition  \ref{theoremk(q)} below. 
 
 To obtain the results for the formula $k(q_C),$ we will list the ternary quadratic forms according to the order of their automorphism groups. Section \ref{ternaryquadraticforms} is mostly dedicated to doing that because it is also of independent interest. In a special case of Theorem \ref{propaut6_12_24} below, we have:
 \begin{thm}
     \label{th: special_case_aut12_24}
     Let $q(x,y,z) = ax^2+by^2+cz^2+2ryz+2sxz+2txy$ be an Eisenstein reduced positive ternary quadratic form. Let us put $u=|\Aut^+(q)|.$ If $3\mid u>6$, then
\sps

\noindent\emph{(i)} $u=12\Leftrightarrow a=b=-2t,\ r=s=0 \textrm{ or } a\neq b,\ b=c=-2r,\ s=t=0.$
\spm

\noindent\emph{(ii)} $u=24\Leftrightarrow a=b=c=\kappa r=\kappa s=\kappa t,$ with $\kappa=2,-3$ or $a=b=c,\ r=s=t=0.$
 \end{thm}

This is the special case of Theorem \ref{propaut6_12_24} below. Moreover, we will discuss the similar results when $3\nmid |\Aut^+(q)|$; cf.\ Theorem \ref{LemAut4_8} below. While we prove these results, the tables  of Theorem 105 of Dickson\cite{dicksonsbook} will be a cornerstone.

Let $R_n(q):=\{(x_1,\dots,x_r)\in\Z^r: q(x_1,\dots,x_r)=n\}$ denote the set of the numbers (or vectors) of representations of an integer $n$ by a quadratic form $q,$ and let $r_n(q)=|R_n(q)|.$

For a positive ternary form $q,$ let 
$$
a(q) \;=\; \max(1, r_4(q), 3r_4(q)-12).
$$
 This quantity is the other cornerstone  in this article because it has a very nice connection with $\Aut(C)$ as follows.
\begin{align}
\label{a(q)_Aut(C)}
      2a(q_C) \;=\; |\Aut(C)|
  \end{align} by equations (3), (26) and (33) of \cite{cas} (cf.\ Remark \ref{remark_C10} below.)

By equation (\ref{a(q)_Aut(C)}), we have that $1/k(q_C) = a(q_C)/|\Aut^+(q_C)|,$ which plays an important role in \cite{cas} because it is a part of the formula of the number of isomorphism classes of (smooth) curves  $C$ of genus 2 lying on a given abelian surface $A.$ 

 Kani computed $a(q_C)$ when $q_C$ is a primitive ternary form (cf.\ Proposition \ref{prop: Corollary_30} below). He mentioned in Remark 31 of \cite{cas} that the computation of $a(q_C)$ is more complicated when $q_C$ is not primitive. Our list makes explicit the formula $a(q_C)$ for any imprimitive ternary form $q_C;$ cf.\ Table \ref{Tab:k(C)}. 

By using the properties of the refined Humbert invariant $q_C$, one can study the elliptic subcovers of degree $n$ of a genus 2 curve $C$; cf.\ (\ref{eq: elliptic_subcover_fact}) below. In Section \ref{s: elliptic_subcovers}, we apply our results to this research area, and we describe the curves $C$ of having an elliptic subcover of degree $n$ for a given $\Aut(C)$ in some cases.
\spm

{\small \noindent\textit{Acknowledgements.}
Part of this project grew out of the PhD thesis of the author. He thanks his supervisor, Prof.\ Ernst Kani for his extensive guidance with particular suggestions on (some parts of) the present article, and sharing of his current (preprint) papers with the author.

The author also wants to thank Eda K{\i}r{\i}ml{\i} for proofreading the present article numerous times and giving many suggestions on it, and  pointing out many errors and typos.}


\section{The refined Humbert invariant}

\label{refinedhumbert}
In this section, we recall the concept of the refined Humbert invariant; introduced by Kani\cite{kani1994elliptic},\cite{MJ}, and recall its significant properties. This is a quadratic form that is intrinsically attached to a principally polarized abelian surface $(A,\theta).$ 

Let $A/K$ be an abelian surface over an algebraically closed field $K$ with $\charec K = 0$. If $\lambda: A\stackrel\sim\rightarrow\hat{A}$ is a principal polarization of $A$, then $\lambda=\phi_{\theta}$ for a (unique) $\theta\in \NS(A)=\Div(A)/\equiv,$ where $\equiv$ denotes numerical equivalence; (see \cite[p.~60]{mumford1970abelian} for the discussion of $\phi$). 

Let $\mathcal{P}(A) \subset \NS(A)$ denote the set of principal polarization on $A$, and let $\theta\in \mathcal{P}(A)$. The quadratic form
$\tilde q_{(A,\theta)}$ of a principally polarized abelian surface $(A,\theta)$ on $\NS(A)$ is defined by the formula
\begin{equation}
\label{eq:qth} 
\tilde q_{(A,\theta)}(D) \;=\; (D.\theta)^2 - 2(D.D),\quad\mbox{for } D\in \NS(A),   
\end{equation}
where $(.)$ denotes the intersection number of divisors. By \cite[p.~200]{kani1994elliptic}, the form $\tilde q_{(A,\theta)}$ induces a positive definite quadratic form $q_{(A,\theta)}$ on the quotient module 
$$
\NS(A,\theta) \;=\; \NS(A)/\Z\theta.
$$  

The quadratic form $q_{(A,\theta)}$ or, more correctly, the quadratic module  $(\NS(A,\theta), q_{(A,\theta)})$ is called the \emph{refined Humbert invariant} of the principally polarized abelian surface $(A,\theta)$; cf.\ \cite{MJ}, \cite{kani2014jacobians}. Let $\rho=\rank(\NS(A))$ be the Picard number of $A$. Thus $q_{(A,\theta)}$ gives rise  to an equivalence class of integral quadratic forms in $\rho-1$ variables since $\NS(A,\theta)\simeq\mathbb{Z}^{\rho-1}.$ Hence, when $\rho=4,$ the (corresponding) refined Humbert invariant is actually a ternary quadratic form.

Recall from \cite{milne1986jacobian} that if $C/K$ is a curve of genus $2$, then its Jacobian $J_C$ is an abelian surface
and there is a divisor $\theta_C$ on $J_C$, called the \textit{theta-divisor} such that $\theta_C$ is a principal polarization in $\NS(J_C)$, and $\theta_C\simeq C$. We then write 
$$
q_C \;:=\; q_{(J_C,\theta_C)}
$$
for its associated refined Humbert invariant.  

One key property of the refined Humbert invariant $q_{(A,\theta)}$ is determining whether $(A,\theta)$ is a Jacobian or not. By Proposition 6 of \cite{MJ}, we have that 
\begin{equation} 
\label{thetaisirreducible}
(A,\theta) \simeq (J_C,\theta_C), \mbox{ for some curve }C/K\; \Leftrightarrow\; 
q_{(A,\theta)}(D)\neq 1, \forall D\in \NS(A,\theta).
 \end{equation}



  When  $A/K$ is a CM (complex multiplication) abelian product surface, i.e., $A\simeq E_1\times E_2$ for some CM elliptic curves $E_1/K, E_2/K$ with $E_1\sim E_2,$ we have that $q:=q_{(A,\theta)}$ is a ternary form. If $(A,\theta) \simeq (J_C,\theta_C)$,  for some curve $C/K$, then $q$ does not represent 1 by equation (\ref{thetaisirreducible}), and hence, we can apply Theorem 3 and Proposition 16 of \cite{cas} to the form $q$. Then we obtain from those results that 
   \begin{align}
   \label{athetadividesAut}
      a(q)  \mid  |\Aut^+(q)|,
  \end{align}
where $\Aut^+(q)=\{\alpha\in\Aut(q): \det(\alpha)=1\}$ (and $a(q)$ is as in the introduction).

The second key property of the refined Humbert invariant $q_C$ is possessing the information about  $\Aut(C)$. The consequence of this property is the key ingredient in this article, and so we now give it as a table explicitly from Theorem 25 of \cite{cas}:

 Let $C/K$ be a curve of genus 2 such that $q_C$ is a binary or a ternary form.  Let $C_n$ denote the cyclic group of order $n.$ Then we have 

 \begin{table}[H]
 {\label{Tab:autC}}
     \centering
      \caption{Automorphism groups of a curve $C$ of genus 2 with $r_4(q_C)$}
     \begin{tabular}{c|c|c}
        $\Aut(C)$ & $a(q_C)$ & $r_4(q_C)$ \\ \hline
 $C_2$ & 1 & 0 \\ 
  $C_{10}$ & 1 & 0 \\ 
  $C_2\times C_2$ & 2 & 2 \\
   $D_4$ & 4 & 4 \\
    $D_6$ & 6 & 6 \\
     $C_3\rtimes D_4$ & 12 & 8 \\ 
 $\GL_2(3)$ & 24 & 12 \\ 
     \end{tabular}
    
     \label{tab:my_label}
 \end{table}

Since this table plays a key role in the present article, we sketch the main ideas of Kani's proof. The first column of this table gives the complete list of possibilities for $\Aut(C)$, for any genus 2 curves $C/K$; cf.\ Theorem 2 of \cite{shaska2004elliptic}. To verify the third column, if $i(\Aut(C))$ denotes the number of involutions of the group $\Aut(C)$, then we have from Theorem 20 of \cite{ESCI} and Proposition of \cite{ESCII} that 
$$
i(\Aut(C)) \;-\; 1 \;=\; r_4(q_C)
$$
(cf.\ Proposition 24 of \cite{cas}). Hence, one can easily verify the third column by the group theory, and the second column follows from the third column immediately.

  \begin{remark} 
  \label{remark_C10}
  If a refined Humbert invariant $q_{(A,\theta)}$ of a principally polarized abelian surface $(A,\theta)$ is a ternary form, then $A\simeq E_1\times E_2,$ for CM elliptic curves $E_1, E_2$. To see this, first remember that $\rank(\NS(A))=4$ in this case. By the structure theorems for $\End(A)$ (cf.\ Proposition IX.1.2 of \cite{van2012hilbert}), it follows that  $A\sim E\times E,$  for some CM elliptic curve $E.$ By Shioda and Mitani's Theorem\cite{MitaniShioda} (or by Theorem 2 of \cite{kani2011products}), it follows that $A$ is a CM abelian product surface.
  
  Therefore, if $q_C$ is a ternary form for some curve $C$ of genus 2, then $\Aut(C)$ cannot be isomorphic to $C_{10}$ by 
   the fact that $J_C$ is simple when $\Aut(C) \simeq C_{10}$; cf.\ \cite[p.~648]{Igusa}.
  Thus, the second line of Table \ref{Tab:autC} cannot occur in the list of the ternary forms $q_C,$ for a given automorphism group $\Aut(C).$ Thus, we are just left with different numbers $a(q_C)$ (and $r_4(q_C)$) for each case in the table,  and so  we have that \begin{align*}
      a(q_C)\; = \; \frac{1}{2}|\Aut(C)|,
  \end{align*} in this case as stated in (\ref{a(q)_Aut(C)}).
  \end{remark}

\section{Ternary quadratic forms} \label{ternaryquadraticforms}

Since the form $q_C$ encodes the information of $\Aut(C)$ (cf.\ Table \ref{Tab:autC}),  our aim is to list the ternary quadratic forms $q_C$ with given $\Aut(C)$. To this end, since $a(q_C)\mid|\Aut^+(q_C)|$ by equation (\ref{athetadividesAut}), we first determine the ternary quadratic forms $q$ according to the order of the group $\Aut^+(q)$ in this section.

For simplicity, we use the abbreviation 
$$
[a,b,c,r,s,t] \;:=\;  ax^2+by^2+cz^2+ryz+sxz+txy \;=\; q(x,y,z)
$$
to denote an \textit{integral ternary quadratic form} $q(x,y,z)$ with $a,b,c,r,s,t\in\Z.$ Note that if all non-diagonal coefficients, namely $r,s,t,$ are even, then $q$ is also an \textit{integral ternary quadratic form in Dickson's sense}\cite{dicksonsbook}.

The properties of an \textit{Eisenstein reduced} form are often used below, so let us recall it here. Firstly, by equivalent ternary forms $q_1$ and $q_2,$ we mean $\GL_3(\Z)$-equivalence, and we denote it by $q_1\sim q_2;$ cf.\ \cite[p.~5]{watson1960integral}. By Theorem 103 of \cite{dicksonsbook}, every positive ternary quadratic form in the sense of Dickson  is equivalent to a form  $q(x,y,z)=ax^2+by^2+cz^2+2ryz+2sxz+2txy$ having the following conditions, where $\epsilon=a+b+2r+2s+2t:$
\begin{align}
\label{condition1}
    &r,\ s,\ t \ \textrm{ are all positive or all are non-negative.}\\
    \label{condition2}
    &a \;\leq\; b \;\leq\; c,\ \epsilon \;\geq\; 0.\\
    \label{condition3}
    &a \;\geq\; |2s|,\ a \;\geq \; |2t|,\ b \;\geq\; |2r|.\\
    \label{condition4}
     &\textrm{ If } a \;=\; b,\ |r| \;\leq\; |s|; \textrm{ if } b\;=\;c,\ |s| \;\leq\; |t|; \textrm{ If }\epsilon \;=\;0,\ a \;+\; 2s \;+ \;t \;\leq\; 0.\\
     \label{condition5}
    & \textrm{ If }a \;=\;-2t,\ s \;=\;0; \textrm{ if }a\;=\;-2s,\ t\;=\;0; \textrm{ if }b \;=\;-2r,\ t\;=\;0.\\
    \label{condition6}
    & \textrm{ if }a\;=\;2t,\ s\;\leq\; 2r; \textrm{ if }a\;=\;2s,\ t\;\leq\; 2r; \textrm{ if }b\;=\;2r,\ t\;\leq\; 2s.
\end{align}

Note that the condition (\ref{condition5}) is valid for $r,s,t\leq 0,$ and the condition (\ref{condition6}) is valid for $r,s,t>0.$ The main point of introducing the definition of a reduced form is the remarkable fact that no two reduced forms are equivalent by Theorem 104 of \cite{dicksonsbook}.

We list the ternary forms $q$ according to the order $|\Aut^+(q)|$ as was mentioned in the introduction. For this, we use the tables of Theorem 105 of Dickson\cite{dicksonsbook}.  Here, he listed the existence of the automorphs of a positive reduced ternary form according to the relations of the coefficients. By calculating the orders of the elements in those tables,  we can see that the automorphs listed in lines $1\textrm{\textendash}7$ of the table on \cite[p.~179]{dicksonsbook} all have order 2, and that those in line 8 all have order 4. Also, the lines  $9\textrm{\textendash}10$ of the table contain the automorphs of order 3 and 2. Similarly, we can see that the automorphs listed in lines $1\textrm{\textendash}7$ of the table on  \cite[p.~180]{dicksonsbook} all have order 2, and that those in lines 8 and 14 all have order 4. Also, the line 15 of the table contains the automorphs of order 4 and 2. In this table, all automorphs of order 3 occur in the lines $9\textrm{\textendash}12$ and 16. These computational results will be frequently used below.

We now prove  the first result related to the discussion above. It basically lists the positive ternary quadratic forms $q$ when  $3\mid |\Aut^+(q)|$, and so it proves Theorem \ref{th: special_case_aut12_24}.

\begin{thm} 
\label{propaut6_12_24}
Let $q(x,y,z) = ax^2+by^2+cz^2+2ryz+2sxz+2txy$ be an Eisenstein reduced positive ternary quadratic form. Let us put $u=|\Aut^+(q)|.$ If $3\mid u,$ then  we have that
\sps

\emph{(i)} $u=6\Leftrightarrow a=b=c,\ r=s=t\neq 0\textrm{ and } r>0\Rightarrow  a\neq 2r \textrm{ and } $  $r<0\Rightarrow a\neq-3s,\ a\neq -2r,$ or $a=b=2r=2s=2t,\ b\neq c,$  or $b=c,\ s=t,\ 2r+s=-b,\ a=-3s,\ a\neq b.$ 
\spm

\emph{(ii)} $u=12\Leftrightarrow a=b=-2t,\ r=s=0 \textrm{ or } a\neq b,\ b=c=-2r,\ s=t=0.$
\spm

\emph{(iii)} $u=24\Leftrightarrow a=b=c=\kappa r=\kappa s=\kappa t,$ where $\kappa=2,-3$ or $a=b=c,\ r=s=t=0.$

\end{thm}

\begin{proof} Assume that $3\mid u.$ Thus,  $\Aut^+(q)$ has an element of order 3. To verify the one direction ($\Rightarrow$) for all assertions we find all reduced forms $q$ satisfying the hypothesis by using  the tables of Dickson\cite{dicksonsbook}, on pages 179 and 180. The other direction ($\Leftarrow$) can be easily derived from the same tables.  
\spm

\noindent{\textbf{Case 1.}} Suppose first that $r,s,t>0.$ 
\sps

\noindent Note that  since all the automorphs of order 3 occur in lines $9\textrm{\textendash}11$ of the table on \cite[p.~179]{dicksonsbook} by what was discussed above, it follows that we have three subcases as follows: $a=b=2r=2s=2t$ or $a=b=c,\ r=s=t$ or $a=b=c=2r=2s=2t.$  We can directly conclude by \cite[p.~180]{dicksonsbook} that $u=24$ if the last subcase, i.e., the case of line 11, holds. Thus, we can eliminate this subcase from now on.

Observe that when $a=b=c,\ r=s=t,$ and $a=2r,$ we see that the case of line 11 holds, which was excluded above. Hence, we assume that $a\neq 2r$ in Subcase 1 below.
\spm

\noindent\textit{Subcase} 1. Suppose that $a=b=c,\ r=s=t,\ a\neq 2r.$
\sps

\noindent We see that the cases of lines $2\textrm{\textendash}4$ and 8 are not possible. Moreover, we see that the cases of lines 1, 7, 9 and 11 do not hold since $a\neq 2r.$  We are thus left with the cases of lines 5, 6 and 10, and so $u=6$ since there is a total of 5 automorphms in those lines except the identity.
 
As in Subcase 1, we assume that $b\neq c$ in Subcase 2 below since the case of line 11 would hold, otherwise.
\spm
 
\noindent\textit{Subcase} 2. Suppose that $a=b=2r=2s=2t,\ b\neq c.$
 \sps
 
\noindent We immediately  see that the cases of lines $2\textrm{\textendash}4$, $6\textrm{\textendash}8$ and $10\textrm{\textendash}11$ are not possible. Hence, this implies that we are just left with the cases of lines  1, 5 and 9, which again shows that  $u=6$ by the same reason in Subcase 1.

Therefore, we have verified the one direction  ($\Rightarrow$) of the assertions (i)-(iii) in Case 1 because we have determined all the reduced forms satisfying the above conditions. Moreover, the other direction ($\Leftarrow$) also follows in this case because we have found $u$ for all possible forms in this case.
 \spm
 
 \noindent\textbf{Case 2.} Assume that $r,s,t\leq0.$
\sps

Note that  since all the automorphs of order 3 occur in lines $9\textrm{\textendash}12$ and 16 of the table on \cite[p.~180]{dicksonsbook} by what was discussed above, it follows that  we have five subcases to analyze. We can readily conclude by \cite[p.~180]{dicksonsbook} that $u=24$ if the case of the line 16 holds, i.e., $a=b=c=-3s,\ r=s=t.$ Thus, we can eliminate this subcase from now on.
 
 Before starting to analyze the remaining four subcases, note that if $r=s=t=0$ and $3\mid u,$ then $a=b=c$ by the assertion of the footnote on p.\ 180 of \cite{dicksonsbook}. Moreover, we have from by \cite{dicksonsbook}, loc. cit., that if $a=b=c$ and $r=s=t=0,$ then $u=24.$ From now on, we assume that at least one of the values $r,s,t$ is not zero.

We clearly see that if the case of line 12 holds, i.e., $a=b=c,\ r=s=t,$ and if $a=-3s,$ then the case of line 16 holds, which was excluded. Hence, we suppose that $a\neq-3s$ in the case that line 12 holds. Moreover, since $q$ is reduced, we see that $a\neq -2r$ by the property (\ref{condition5}) in this case, and we suppose this condition as well.
\spm

\noindent \textit{Subcase} 1. Assume that $a=b=c,\ r=s=t,\ a\neq -3s,\ a\neq -2r.$
\sps

 Since $a\neq -3s,$ we see that $a+2s+t\neq0,$ and thus the cases of lines 4 and 11 are not possible. Moreover, it is easy to see that if $a+b+2r+2s+2t=0,$ then $a=-3s,$ which is not possible, and thus we can discard the cases of lines $7$ and $13\textrm{\textendash}15.$ 
 
 Note that $r\neq0;$ otherwise, $r=s=t=0,$ which is not possible as was mentioned above. Hence, we are left with the cases of lines $5\textrm{\textendash}6$ and 12. Thus, it follows that $u=6$ since there are totally 5 automorphms in those lines except the identity.
 
 Since we analyzed the case of line 12, we may eliminate this case for further analysis. Also, we observe that if the case of line 11 holds, i.e., $b=c,\ s=t,\ b+2r+s=0,\ a=-3s$ and if $a=b,$ then $r=s,$ and so the case of line 16 holds, which was excluded. Hence, we  suppose that $a\neq b$ in the case that line 11 holds as follows.
 \spm
 
 \noindent\textit{Subcase} 2. Assume that $b=c,\ s=t,\ b+2r+s=0,\ a=-3s, \ a\neq b.$
\sps

\noindent We have that $s\neq 0 \textrm{ and }t\neq 0,$ so we can discard the cases of lines $1\textrm{\textendash}3$ and $8\textrm{\textendash}10.$ Since $a\neq b,$ we can eliminate the cases of lines 5, 7 and $13\textrm{\textendash}15.$ We are thus left with the cases of lines 4, 6 and 11, and this implies that $u=6$ by the same reason in Subcase 1.
 
 We thus eliminate the case of line 11 for the rest of the proof.
\spm

\noindent\textit{Subcase} 3. Assume that  $a=b=-2t,\ r=s=0.$
\sps

\noindent We first see that the cases of lines 1, 3, 8 and 10 are not possible since $r=s=0.$ Moreover, we observe that $a+b+2r+2s+2t\neq 0$ when $r=s=0$ and $a=-2t.$ Thus, we can eliminate the cases of lines 7 and $13\textrm{\textendash}15.$ Recall from above that we assumed at least one of the values $r,s,t$ is nonzero, and so $t\neq0,$ and thus the case of line 6 is not possible. We clearly see that $a+2s+t\neq 0$ in this subcase, and so we can discard the case of line 4. Since we excluded the cases of lines $11\textrm{\textendash}12$ and 16, we are thus left with the cases of lines 2, 5 and 9 of the table, and they give six automorphs including the identity. By the footnote of \cite[p.~180]{dicksonsbook}, there are precisely six more automorphs in our case, and hence  $u=12$ in this subcase.

In the following subcase, we analyze the case of the line 10, i.e., $b=c=-2r,\ s=t=0.$  Since $q$ is reduced, we can assume that $a\neq b$ by the property (\ref{condition4}).
\spm

\noindent\textit{Subcase} 4. Assume that $b=c=-2r,\ s=t=0,\ a\neq b.$ 
\sps

 As in Subcase 3,  we similarly see that we are just left with the cases of lines 3, 6 and 10, and they give six automorphs including the identity. As in Subcase 3, there are precisely six more automorphs by the footnote of \cite[p.~180]{dicksonsbook}, and hence  $u=12.$

 Therefore, we have verified the one direction ($\Leftarrow$) of all assertions (i)-(iii) by the Subcases. Moreover, since we analyzed all the possible reduced forms $q$'s such that $3\mid u$ in Cases 1 and 2, the other direction follows from the results proven in the Subcases. 
\end{proof}

Let $\cont(q)$ denote the greatest common divisor of the coefficients of an integral quadratic form $q.$ When $A$ is a CM (abelian) product surface, we have that 
\begin{equation}
\label{content1or4}
 \textrm{either }   \cont(q_{(A,\theta)}) \;=\; 1 \textrm{ or } \cont(q_{(A,\theta)}) \;=\; 4, \textrm{ for any } \theta\in\mathcal{P}(A);
\end{equation}
by Theorem 20 of \cite{kani2014jacobians} and Proposition 19 of \cite{Kir} (cf.\ Proposition 15(ii) of \cite{Kir}).
Moreover, if $q\sim q_C,$ for some curve $C$ of genus 2, we have from \cite[p.~24]{MJ} that 
\begin{equation}
\label{0,1MOD4}
    q(x,y,z) \;\equiv\; 0,1\Mod{4}, \textrm{ for all }x,y,z\in\Z.
\end{equation}

Therefore, it suffices to determine $r_4(q)$, for ternary  forms $q$ in the case that $q$ satisfies the conditions in equations (\ref{content1or4}) and (\ref{0,1MOD4}). It is useful to know that it is done  in the primitive case, i.e., when $\cont(q)=1$ by E. Kani:

\begin{prop}
    \label{prop: Corollary_30}
 Let $q_C$ be a primitive ternary form. Then $q_C \sim q$, where $q$ is Eisenstein reduced. Moreover, if we write $q(x, y, z)=a x^2+b y^2+c z^2+2 r y z+2 s x z+2 t x y$, then we have that
$$
(a(q_C), r_4(q_C))  \;=\; \begin{cases} (1,0) & \text { if } a \neq 4 \\ (2,2) & \text { if } a=4 \text { and } b \neq a \\ (6,6) & \text { if } a=b=4 \text { and }(r, s, t)=(2,2,2) \text { or }(0,0,-2) \\ (4,4) & \text { otherwise. }\end{cases}
$$
\end{prop}

\begin{proof}
    This follows from Corollary 30 of \cite{cas} and Table \ref{Tab:autC}.
\end{proof}

Thus, we want to determine $r_4(q)$, for imprimitive ternary forms satisfying the conditions in (\ref{content1or4}) and (\ref{0,1MOD4}). However, the situation is more complicated in this case, and so it needs considerable work.
\spm

\noindent\textbf{Notation.} Let $q_{a,b,c}$ denote the diagonal ternary quadratic form $ax^2+by^2+cz^2.$ Let us put 
\begin{align*}
    q_{1,c}=&[4,4,c,4,4,4],\quad q_{2,c}=[4,4,c,0,0-4],\quad q_{3,c}=[4,4,c,-4,-4,0],\\
    q_{4,c}=&[4,c,c,-4,0,0],\quad q_{5,c}=[4,c,c,-c,0,0],\quad q_{6,c}=[4,4,c,0,-4,0].
\end{align*} 
It is straightforward to check that all the ternary forms $q_{i,c}$ are (Eisenstein) reduced forms, for $i=1,2$ provided  $c\geq4.$ Similarly, it is a simple calculation to see that all the ternary forms $q_{i,c}$ are reduced forms, for $3\leq i\leq 6$ provided  $c>4.$ 

If we apply Theorem \ref{propaut6_12_24} to the ternary quadratic forms having given properties, we obtain that:
\begin{cor}
\label{Aut6_12_24forimprimitiveform}
Let $q$ be a positive ternary quadratic form such that $\cont(q)=4$ and $r_4(q)>0.$ Let $u=|\Aut^+(q)|$.
\sps

\emph{(i)} If $u=24,$ then $q\sim q_{1,4}$ or $q\sim q_{4,4,4}.$
\spm

\emph{(ii)} If $u=12,$ then  $q\sim q_{2,c}$ or $q\sim q_{5,c},$ with $4\mid c>4.$ 
\spm

\emph{(iii)} If $u=6,$ then $q\sim q_{1,c},$ with $4\mid c>4.$ 
\end{cor}

\begin{proof} Since $\cont(q)=4,$ it follows that $q$ is an integral quadratic form in the sense of Dickson, and by replacing $q$ by an equivalent form, we may assume that $q$ is an Eisenstein
reduced form; cf.\ \cite{dicksonsbook}, Theorem 103. Let us put $q=ax^2+by^2+cz^2+2ryz+2sxz+2txy.$ Since $\min(q)=a$ by Theorem 101 of \cite{dicksonsbook}, and since $\cont(q)=4$ and $r_4(q)>0,$ it follows that $a=4.$

To prove the assertion (i), assume that $u=24.$ We thus have by Theorem \ref{propaut6_12_24} that $b=c=2r=2s=2t=4$ or  $b=c=4,\ r=s=t=0,$ which proves that $q= q_{1,4}$ or $q= q_{4,4,4}.$

 Now suppose that $u=12.$ By Theorem \ref{propaut6_12_24}, we have that $b=-2t=4,\ r=s=0,$ which implies that $q= q_{2,c},$ with $4\mid c$  or that $b=c=-2r,\ s=t=0,\ b\neq 4,$ which implies that $q= q_{5,c},$ with $4\mid c>4$ (since $\cont(q)=4$). Thus, this proves the second assertion.

   Lastly, assume that $u=6.$ By Theorem \ref{propaut6_12_24} again,  since $4\neq -3s,$ we have only two possible cases that $4=b=c,\ r=s=t\neq 0\textrm{ and } r>0\Rightarrow4\neq 2r \textrm{ or }r<0\Rightarrow4\neq-3s\textrm{ and } 4\neq -2r,$ or $4=b=2r=2s=2t\textrm{ and } b\neq c.$ Suppose that the former case holds. Since $\cont(q)=4$ and  since $4=b\geq|2r|$ by the property (\ref{condition3}), this case is not possible. We next suppose that the latter case holds, and thus we obtain that $4=b=2r=2s=2t$ and $c\neq 4,$ and thus $q=q_{1,c},$ for $4\mid c\geq 8$ since $\cont(q)=4$ and $c\geq b$ by the property of (\ref{condition2}). Thus, the last assertion also follows.
\end{proof}


We now  calculate the number $r_4(q)$ for some ternary forms $q$ to use Table \ref{Tab:autC} as aimed. 

 \begin{lem}
 \label{r4(q)}
 We have the following numbers $r_4(q)$ of representations of $4$ by given ternary forms $q.$
 \sps
 
 \emph{(i)} $r_4(q_{4,4,4})=6,$ $r_4(q_{1,4})=12,$  and $r_4(q_{1,c})=6,$ for $c>4$ and $4\mid c.$
 \spm

 \emph{(ii)} $r_4(q_{2,4})=8 \textrm{ and }r_4(q_{2,c})=6,$ where $c>4.$
 \spm

 \emph{(iii)} $r_4(q_{3,c})=4,$ for $c>4$ and $4\mid c.$ 
\end{lem}

\begin{proof}
(i)  First note that $q_{1,4}=2[(x+y)^2+(y+z)^2+(x+z)^2],$ so it is easy to see that
$$
R_4(q_{1,4})=\{(\pm1,0,0),(0,\pm1,0),(0,0,\pm1),\pm(-1,1,0),\pm(-1,0,1),\pm(0,-1,1)\},
$$
and hence, $r_4(q_{1,4})=12.$ We next easily see that 
$$
R_4(q_{4,4,4})=\{(\pm1,0,0),(0,\pm1,0),(0,0,\pm1)\},
$$
and so, $r_4(q_{4,4,4})=6.$ In a similar way, when $c>4,$ we observe that 
\begin{equation}
\label{eq: expression_q_1c}
    q_{1,c}=2[(x+y)^2+(y+z)^2+(x+z)^2]+(c-4)z^2.
\end{equation}
 Also, if $(x,y,z)\in R_4(q_{1,c})$ and $z\neq 0,$ then $z=\pm 1$ and $x+y=y+z=x+z=0$ since $c\geq8.$ But, this is impossible since this would imply that $x=-y$ and $x=y.$ Hence, if $(x,y,z)\in R_4(q_{1,c}),$ then $z=0,$ and thus $r_4(q_{1,c})=r_1(x^2+xy+y^2)$ since $x^2+xy+y^2$ is a positive binary form. Also, it is known that $r_1(x^2+xy+y^2)=6,$ namely $\{\pm(0,1),\pm(1,0),\pm(1,-1)\},$ and thus $r_4(q_{1,c})=6,$ so the first assertion follows. 

(ii) First consider $q_{2,4}=4x^2+4y^2+4z^2-4xy=2(x-y)^2+2x^2+2y^2+4z^2.$ If $(x,y,z)\in R_4(q_{2,4})$ and $z\neq 0,$ then it follows that $(x,y,z)=\pm(0,0,1).$ Suppose now that $z=0,$ and so we should find the number $r_1(x^2-xy+y^2),$ but it is known that   $x^2-xy+y^2=1$ has 6 solutions (as in the first part). Therefore, we have totally 8 solutions, and so $r_4(q_{2,4})=8.$

Next, suppose that $c\geq 5.$ Observe that 
\begin{equation}
    \label{eq: expression_q_2c}
    q_{2,c}=4x^2+4y^2+cz^2-4xy=3x^2+cz^2+(x-2y)^2.
\end{equation}
Also, since $c\geq 5,$ if $(x,y,z)\in R_4(q_{2,c}),$ then $z=0.$ Then, we can see that $r_4(q_{2,c})=r_1(x^2-xy+y^2)=6$ as was mentioned above. More precisely, we have $R_4(q_{2,c})=\{\pm(0,1,0),\pm(1,0,0),\pm(1,1,0)\},$ and so $r_4(q_{2,c})=6,$ which proves this assertion.

   (iii)  Now, observe that $q_{3,c}=4x^2+4y^2+cz^2-4yz-4xz=2[(x-z)^2+(y-z)^2+x^2+y^2]+(c-4)z^2.$ Also, since $c\geq 8$ by the hypothesis, if $(x,y,z)\in R_4(q_{3,c}),$ then we see that $z=0$ as in part (i). Then, we can see that $R_4(q_{3,c})=\{\pm(0,1,0),\pm(1,0,0)\},$ and it proves that $r_4(q_{3,c})=4.$
\end{proof}

Before we list the ternary forms $q$ when $|\Aut^+(q)|$ is 2, 4 or 8, we introduce more notations which will be often used below. 
\spm

\noindent\textbf{Notation.} Assume that $b\leq c$ and $b\geq |2r|$ below. \begin{align*} 
    q_1=&[4,b,c,0,0,-4] \textrm{ with } b>4,  \quad   q_2=[4,b,b,2r,4,4] \textrm{ with } b>4,\ r>1, \\
   q_3=&[4,b,c,0,-4,0] \textrm{ with } b\neq c,\quad q_4=[4,b,c,-b,-4,0] \textrm{ with } b>4,\ b\neq c, \\
    q_5=&[4,b,c,2r,4,4] \textrm{ with } b\neq c,\ r>1,\ q_6=[4,b,c,2r,0,-4] \textrm{ with }\\ & 4<b\neq -2r,\ r<0,\  q_7=[4,b,c,2r,-4,0] \textrm{ with } b\neq c,\ b\neq -2r,\ r<0.
\end{align*}

The following lemma is a straightforward verification.
\begin{lem}
\label{reducedforms_q7}
    All forms $q_i$ are Eisenstein reduced forms, for $1\leq i\leq 7.$
\end{lem}
\begin{proof} In this proof, we use the notations in the definition of the Eisenstein reduced form; cf.\ equations (\ref{condition1})-(\ref{condition6}). So, by the nondiagonal coefficients, we mean $r, s, t$ respectively.
    It is clear that the properties (\ref{condition1}), (\ref{condition3}) and the first inequality in the property (\ref{condition2}) hold for all forms. It is also clear that $\epsilon>0,$ for the forms $q_i,$ for $1\leq i\leq 5.$ Since $b\geq |2r|,$ this also holds for the forms $q_i,$ for $i=6, 7.$ Hence, the second inequality in the property (\ref{condition2}) holds for all forms. 
    
    Moreover, since $\epsilon>0,$ the last inequality in the property of (\ref{condition4}) does not apply in our case. To check the other inequalities in the property of (\ref{condition4}), we first suppose that $b=4.$ Then we have the $q_i,$ for $i=3, 5 \text{ or } 7.$ It is clear that the other inequalities in that property hold for $q_3.$ If $q_i,$ for $i=5\text{ or }7,$ then it also follows since $4=b\geq|2r|.$ We next suppose that $b\neq 4.$ Then, the rest cases are $q_i,$ for $i=1, 2 \text{ or } 6.$ But, we clearly have that the second inequality in the property (\ref{condition4}) holds.

    We lastly check that  the forms $q_i,$ for $i=1,3,4,6\text{ or }7$ to show that the property (\ref{condition5}) holds and that the forms $q_i,$ for $i=2\text{ or }5$ to show that the property (\ref{condition6}) holds. For the first one, we easily see that if $t=-2,$ then $s=0$ in all cases which were mentioned above. Also, we see that if $s=-2,$ then $t=0$ in those cases. We now show that if $b=-2r,$ then $t=0.$ Note first that $t=0$ in the case that $q_i,$ for $i=3,4.$ So, the case that $q_3$ and $q_4$ trivially satisfy this condition. Since the forms $q_6$ and $q_7$ do not satisfy that $b=-2r,$ these forms vacuously satisfy that condition. Also, $q_1$ satisfies this condition trivially since $r=0.$ Therefore, the property (\ref{condition5}) holds for the forms $q_i,$ for $i=1,3,4,6\text{ or }7.$

    The property (\ref{condition6}) holds for the forms $q_2$ and $q_5$ since $2r\geq 4.$ Hence, we proved that all forms satisfy all properties of the reduced forms, and so the assertion follows.
\end{proof}

Here is the list of the ternary forms $q$ having exact properties when $|\Aut^+(q) | = 2, 4, \text{ or } 8$.

\begin{thm}
\label{LemAut4_8}
Let $q=4x^2+by^2+cz^2+2ryz+2sxz+2txy$ be an Eisenstein reduced positive ternary quadratic form with $q(x,y,z)\equiv 0,1\Mod{4},$ for all $x,y,z\in\Z.$ Let $3\nmid u=|\Aut^+(q)|.$ Then we have that 
\sps

\emph{(i)} $u=8$ $\Leftrightarrow$ $q=q_{3,c},$ with $c>4,$ or $q=q_{4,4,c},$ with $c>4,$ or $q=q_{4,b,b},$ with $b>4,$ or $q=[4,b,b,2,4,4].$
\spm

\emph{(ii)} $u=4$ $\Leftrightarrow$ $q\in\{q_1,q_2,q_3,q_4\},$ or $q=q_{4,b,c},$ with $b>4,\ b\neq c$ or $q=[4,b,c,-b,0,0],$ with $b\neq c,$ or $q=[4,b,b,2r,0,0],$ with $4<b\neq-2r\neq 0,$  
  or  $q=[4,b,c,2,4,4],$ with $4<b\neq c.$
  \spm

  \emph{(iii)} $u=2$ $\Leftrightarrow$ $q\in\{q_5,q_6,q_7\}.$ 
\end{thm}

\begin{proof} 
We have from the property of (\ref{condition3}) that $4\geq|2s|,|2t|$ since $q$ is reduced. Moreover, we observe that $s$ and $t$ are even numbers. Indeed, since $q(0,1,0)=b\equiv 0,1\Mod{4}$ and $q(1,1,0)\equiv b+2t\equiv 0,1\Mod{4},$ it follows that $t\neq\pm1.$ By considering $q(0,0,1)$ and $q(1,0,1),$ we similarly obtain that $s\neq\pm1,$ and thus $s,t\in\{0,\pm2\}.$ Also, by the property of (\ref{condition5}), we have that $s=-2\Rightarrow t=0$ and $t=-2\Rightarrow s=0.$ Hence, there are four possible cases for the pairs $s,t.$ To verify the assertions, we use the tables of Theorem 105 of \cite{dicksonsbook} as was done in the proof of Theorem \ref{propaut6_12_24}. We first show that many cases of those tables can be discarded for four possible cases.

Note that since $3\nmid u,$ we know that there is not an automorph of order 3 of $\Aut^+(q).$  Therefore, the cases of lines $9\textrm{\textendash}11$ of table on \cite[p.~179]{dicksonsbook} and the ones of lines $9\textrm{\textendash}12$ and 16 of the table on  \cite[p.~180]{dicksonsbook}  can be eliminated since  those lines have some automorphs of order 3 by what was discussed above. 
\spm

\noindent\textbf{Case 1.} Assume that  $t=-2$ and $s=0.$
\sps

\noindent We immediately see that the cases of lines 1, $3\textrm{\textendash}4,$ 6, 8 and 15 of the table on \cite[p.~180]{dicksonsbook} are not possible. We next observe that $b+2r+2s\neq 0$ by the property of (\ref{condition5}), and  we are thus left with the possible cases of lines 2 and 5.

Now note that the case of line 2 surely holds. If the case of line 5 held as well, it would follow that the case of line 9 holds, which is not possible by what was mentioned above. Hence, only the case of line 2 holds, which has only one automorph. Therefore, if $r\neq0,$ then $u=2$  and $q=q_6,$ with $b\neq-2r$ (cf.\ property (\ref{condition5})) in this case. Note that $b\neq 4$ in this case by the property of (\ref{condition4}). Also, if $r=0,$ then it follows by the footnote of \cite[p.~180]{dicksonsbook} that $u=4$ and $q=q_1,$ with $b>4.$ Thus, we have determined the reduced forms satisfying the above conditions,  and we have proved that $u=4$ or $u=2$ in those cases, and thus the assertions are verified in this case.
\spm

\noindent\textbf{Case 2.} Suppose that $s=t=0.$
\sps

\noindent We immediately observe that the cases of lines $1\textrm{\textendash}2,$ 4 and 8 are impossible. Moreover, we see that $4+b\neq -2r$ since $q$ is reduced; cf.\ the property (\ref{condition3}). This implies that $4+b+2r+2s+2t\neq0,$ and thus we can eliminate the cases of lines 7 and $13\textrm{\textendash}15.$   We are thus left with the possible cases of lines 3 and $5\textrm{\textendash}6$ of the table of \cite[p.~180]{dicksonsbook}.
\spm

\noindent\textit{Subcase} 1. Suppose that $r=0.$
\sps

\noindent Firstly, we can discard the case of line 3. Secondly, if $b=c=4,$ then the case of line 12 holds, which was eliminated above. Hence, we may assume that $b\neq 4 \text{ or } b\neq c.$

Observe next that since $r=s=t=0,$ we can derive three new automorphs from each one listed by the footnote of \cite[p.~180]{dicksonsbook}. Thus, if $b\neq 4,\ b\neq c,$ then we can eliminate the cases of lines $5\textrm{\textendash}6,$ and we don't have any case of line in the table. Hence, there is only one automorph, namely identity as well as three new automorphs derived from the identity. Thus, $u=4$  and $q=q_{4,b,c},$ with $b>4,\ b\neq c$ in this case.

If  $\textrm{ either } b=4 \textrm{ or }b=c,$ then there is only one case of line 5 or 6 respectively, and this line has only one automorph. Hence, we derive $2\times3=6$ new automorphs from this automorph and the identity. Thus, $u=8$ in those cases. Also, we have that $q=q_{4,4,c},$ with $c>4$ or $q=q_{4,b,b},$ with $b>4$ in those cases. 
\spm

\noindent\textit{Subcase} 2. Suppose that $r\neq 0.$
\sps

\noindent Since the case of line 5 is not possible, we are thus left with the cases of lines 3 and 6. Note that if both lines hold at the same time, then the case of line 10 also holds, which is not possible as was discussed above.  Thus, we can conclude that if  only the case of line 3  holds, i.e., $b=-2r$ and $b\neq c,$ then $q=[4,b,c,-b,0,0],$ with $b\neq c.$ We have that line 3 has only one automorph except the identity. When $s=t=0,$ there are precisely 2 new automorphs by the footnote of \cite[p.~180]{dicksonsbook}, and thus $u=4.$ We also conclude that if only the case of 6  holds, i.e., $b=c$ and $b\neq-2r,$ then $q=[4,b,b,2r,0,0],$ with $b\neq-2r,$ and thus $u=4$ again by the same reason.  Note that $b\neq 4$ in both cases since $q$ is reduced; cf.\ (\ref{condition4}).

Therefore, we have found in this case that   $q$ is one of the five forms which were found above.
For each case, we have found $u$ as stated, and thus all assertions are verified in this case.
\spm

\noindent\textbf{Case 3.} Assume that  $s=-2$ and $t=0.$
\sps

\noindent We immediately see that the cases of lines 2 and 6 are not possible. Since we see that $b\neq c$ by the property of (\ref{condition4}),
 we are left with the cases of lines 1, $3\textrm{\textendash}5$ and $7\textrm{\textendash}8.$
 \spm

\noindent\textit{Subcase} 1. Suppose that $b=-2r.$
\sps

\noindent We see that if $a=b,$ then we are left with the cases of lines 1, $3\textrm{\textendash}5$ and $7\textrm{\textendash}8,$ which implies that $q=q_{3,c},$ with $c>4.$ Since there are precisely 7 automorphs in those lines except identity, it follows that $u=8.$ We also see that if $a\neq b,$ then we are just left with the cases of lines 1 and $3\textrm{\textendash}4,$ which implies that $q=q_4.$ Since there are 3 automorphs in those lines, it follows that $u=4.$
\spm

\noindent\textit{Subcase} 2. Suppose that $b\neq-2r.$
\sps

\noindent We immediately see that the cases of lines $3\textrm{\textendash}4$ and $7\textrm{\textendash}8,$  can be discarded. In addition, we can eliminate the case of line 5 because it would imply that $b=-2r,$ which is not possible. Hence, we are left with the case of line 1. If $r\neq0,$ then $q=q_7.$ Since there is only one automorph in this line, it follows that $u=2.$ If $r=0,$ then $q=q_3$ and we can derive 2 more automorphs by the footnote of \cite[p.~180]{dicksonsbook}, and thus $u=4.$  Therefore, the assertions have been verified in Case 3.
\spm

\noindent\textbf{Case 4.} Suppose that $s=t=2.$
\sps

\noindent We can immediately eliminate the case of line 4 of the table of \cite[p.~179]{dicksonsbook}. We also discard the case of line 5 since it would imply the case of line 9, which is impossible.
\spm

\noindent\textit{Subcase} 1. Assume that $b=c.$
\sps

\noindent We analyze this subcase whether $r=1$ or not. If $r=1,$ then we clearly have that the cases of lines $1\textrm{\textendash}3$ and $6\textrm{\textendash}8$ hold, which implies that $q=[4,b,b,2,4,4].$ Since there are precisely 7 automorphs except the identity, $u=8.$ Next, suppose that $r\neq1.$ We then clearly eliminate the cases of lines $2\textrm{\textendash}3$ and 8. Hence, we are left with the cases of lines 1 and $6\textrm{\textendash}7,$ which implies that $q=q_2$ since $b\neq4$ by the property of (\ref{condition4}). Since there are 3 automorphs in those lines, it follows that $u=4.$
\spm

\noindent\textit{Subcase} 2. Assume that $b\neq c.$
\sps

\noindent We immediately eliminate the cases of lines $6\textrm{\textendash}8.$ Thus we are left with the cases of lines $1\textrm{\textendash}3$, and we see  that $r=1\Leftrightarrow q=[4,b,c,2,4,4]$ and  $r\neq 1 \Leftrightarrow q=q_5.$ In the former situation, we have precisely 3 automorphs in the lines $1\textrm{\textendash}3,$ and in the latter situation, we have just one automorph in the line 1. Hence, we obtain that $u=4$ or $u=2$ respectively. Note that $b\neq4$ in the former situation since $q\equiv0,1\Mod{4}.$

Therefore, all assertions have been verified in all cases.
\end{proof}

\begin{cor} Let $q$ be a positive ternary quadratic form in Dickson's sense, and let $u:=|\Aut^+(q)|.$ If $3\mid u,$ then $u\in\{6,12,24\}.$ In addition, if $\min(q)=4$ and $q\equiv0,1\Mod{4},$ then 
\begin{equation}
   u\in \{1,2,4,6,8,12,24\}. \label{Autf}
\end{equation}
\end{cor}

\begin{proof}
As in the proof of Corollary \ref{Aut6_12_24forimprimitiveform}, we may assume that $q$ is a reduced form. Thus, the first statement directly follows from Theorem \ref{propaut6_12_24}, and in particular, equation (\ref{Autf}) holds when $3\mid u.$  We next suppose that $3\nmid u.$ Since $\min (q)=4,$ we have that $q(1,0,0)=4$ by Theorem 101 of \cite{dicksonsbook}, and so  equation (\ref{Autf}) holds by Theorem \ref{LemAut4_8}.
\end{proof}

Recall that Smith\cite[p.~278]{smith} and  Dickson\cite[p.~179]{dicksonsbook} stated (without proving) equation (\ref{Autf}) for a positive ternary form in Dickson's sense. 

\section{Lists of the ternary quadratic forms}

\label{lists_ternary_forms}

In the previous section, we have determined ternary quadratic forms based on the order of their automorphism groups. We now determine which of those forms is equivalent to a refined Humbert invariant by using the main results in \cite{Kir} and \cite{refhum}. While this task is quite easy in the imprimitive case, it is highly complicated in the primitive case. In the latter situation, our method uses the theory of \textit{genus equivalence} for ternary quadratic forms for some cases. Then, we satisfy a proof of Theorem \ref{firstmainresultPART2} by applying the results in Section \ref{ternaryquadraticforms}.

\begin{prop} \label{allformsarehumbert}
Let $4\mid c>0.$ Then $q_{i,c}$ is equivalent to a refined Humbert invariant $q_C,$ for some curve $C$ of genus $2,$ for $1\leq i\leq 4$ and $i=6.$ Moreover, if $\cont(q_j)=4,$ for $1\leq j\leq7,$ then $q_{j}$ is equivalent to a refined Humbert invariant $q_C,$ for some curve $C$ of genus $2.$  
\end{prop}

\begin{proof}
First of all, note that all ternary quadratic forms in both statements are positive definite. Secondly, we see that all forms have a nondiagonal coefficient $4$ or $-4.$ Since the content of those forms is 4, we thus obtain that the forms $\frac{1}{2}q_{i,c}$ and the forms $\frac{1}{2}q_j$ are improperly primitive. Moreover, since the leading coefficient of those forms is 4, it follows from Theorem 2 of \cite{Kir} that all forms are equivalent to a refined Humbert invariant. Moreover, it is clear that they don't represent 1, and the proposition follows from the assertion (\ref{thetaisirreducible}).
\end{proof}

We continue to determine which of the forms discussed above are equivalent to a refined Humbert invariant when they are primitive. In this case, its proof is more complicated, and it needs arguments related to the theory of ternary forms. To this end, we recall some basic concepts here. Let $f$ be a positive primitive ternary quadratic form. By $\disc(f)$, we mean its \textit{discriminant} as defined in \cite[p.~2]{watson1960integral} (or \cite[p.~335]{brandt1951zahlentheorie}). Recall that $f$ comes equipped with two \textit{basic genus invariants} $I_1(f)$ and $I_2(f)$ introduced by Brandt\cite[p.~316]{brandt1952mass}, \cite[p.~336]{brandt1951zahlentheorie}. Moreover, recall that $f$ has \textit{reciprocal} as defined by Brandt\cite[p.~336]{brandt1951zahlentheorie}; this is denoted by $F_f.$ Depending on the basic genus invariants, we determine the set of  the \textit{assigned characters} $X(f)$ and $X(F_f)$ of $f$ and $F_f;$ cf.\ \cite{brandt1951zahlentheorie}. Let us put 
$$
\chi_{\ell}(x) \;=\; \left(\frac{x}{\ell}\right),
$$ for $x$ prime to $\ell$, where $\ell$ is an odd prime, and
let us define 
$$
\chi_{-4}(x) \;=\; \left(\frac{-4}{x}\right) \;=\; (-1)^{(x-1)/2} \text{ and } \chi_8(x) \;=\; \left(\frac{8}{x}\right) \;=\; (-1)^{(x^2-1)/8},
$$
for odd $x,$ as in
 \cite{brandt1951zahlentheorie}. Here,  the symbol $\left(\frac{.}{.}\right)$ is the the Kronecker-Jacobi symbol.

\begin{prop} 
\label{all_primitiveforms_humbert}
    We have that $f_1:=q_{1,c},\ f_2:=q_{2,c},\ f_3:=q_{3,c'}\ f_6:=q_{6,c},\text{ and } f_4:=q_{4,4,c}$ are equivalent to a form $q_C,$ for some genus $2$ curve $C,$ for all $c, c'>1,$ $c\equiv1\Mod{4}$ and $c'\equiv1\Mod{8}.$
\end{prop}

\begin{proof}
    To prove this assertion we  show that the $f_i$ satisfy the conditions in Theorem 20(iii) of \cite{refhum}. Then, the assertion follows from this result and Theorem 1 of \cite{refhum} (and equation (\ref{thetaisirreducible})).

    We clearly see that those forms are primitive positive and satisfy the condition (\ref{0,1MOD4}), for any $c, c'$. We next claim  that the following conditions hold for the $f_i$.
    \begin{align}
    \label{chi(f)=1}
        \chi(f) &= 1, \text{ for every assigned character } \chi \in X(f).\\\label{I_2(f)=1}
 \left(\frac{I_2(f)}{m}\right)&=1, \text{ for some } m\in R(F_f) \text{ with } \gcd(m, 2I_2(f)) = 1.
    \end{align}

In order to prove these equations for $f_i$'s, we find its basic genus invariants and one suitable element in $R(F_{f_i})$. In this regard, the table below gives all the necessary information on the forms.

The second column of the following table shows the \textit{adjoint forms} $|I_1(f)|F_f$, and the third and the fifth columns are the absolute value of the genus invariant $I_1$ and the genus invariant $I_2$ respectively. The fourth column shows the discriminants and the last column is an exact number represented by the reciprocal forms.

  \begin{table}[H]\centering 
  \begin{threeparttable}
    \caption{Calculational properties of exact ternary quadratic forms}{\label{Table: in_proof}}
     \begin{tabular}{|c|c|c|c|c|c|} 
 \hline
  $i$ & $|I_1(f_i)|F_{f_i}$ &   $|I_1(f_i)|$ & $\disc(f_i)$ & $I_2(f_i)$ & $R(F_{f_i})$\\ [0.5ex] \hline
  
 $1$ &$16[c-1,c-1,3,1,1,c-2]$ & $16$ & $16(4-3c)$ & $4-3c$ & $3$ \\ \hline
  $2'$ &$16[c,c,3,0,0,c]$ & $16$ & $-48c$ & $-3c$ & $4c+3$ \\ \hline
  $2''$ &$48[c/3,c/3,1,0,0,c/3]$ & $48$ & $-48c$ & $-c/3$ & $1$ \\ \hline
  $3$& $32[\frac{c'-1}{2}, \frac{c'-1}{2},2,2,2,1]$ & $32$ & $64(2-c')$ & $2-c'$ & $c'$ \\ \hline
  $4$ & $16[c,c,4,0,0,0]$ & $16$ & $-64c$ & $-4c$ & $c+4$ \\ \hline
  $6$ & $16[c,c-1,4,0,4,0]$ & $16$ & $64(1-c)$ & $4(1-c)$ & $c$ \\ \hline

\end{tabular}

    \begin{tablenotes}
      \small
      \item \textbf{Notation:} In the table, when $3\mid c$, we denote $f_2$ with $f_{2''}$, otherwise, we denote it with $f_{2'}$.
    \end{tablenotes}
  \end{threeparttable}
\end{table}

To verify the table, recall first that  adjoint $\adj(f)$ of $f$ is defined by the formula: 
$$
A(\adj(f)) \;=\; -2\adj(A(f)),
$$
where $A(f)$ is the \textit{coefficient matrix} of $f$ as defined in \cite[p.~2]{watson1960integral} ($\adj(M)$ is the usual adjoint of a  matrix $M$). Thus, the second column is just an easy matrix calculation.  We then easily see from the second column that $|I_1(f_i)|=16,$ for $i=1,2', 4,6$ and $|I_1(f_3)|=32$ since $|I_1(f)|$ is equal to the content of the adjoint form of $f$ (and since $c\equiv 1\Mod{4}$). Also, we see that $|I_1(f_{2''})|=48$,  
and so the third column is verified.

 Since we have a relation
 \begin{equation}
 \label{eq: invariant_discriminant}
     I_1^2(f)I_2(f) \;=\; 16\disc(f)
 \end{equation}
  by \cite[p.~336]{brandt1951zahlentheorie}, it suffices to find the discriminants of the forms $f_i$ to determine the second genus invariant $I_2.$ By the straightforward matrix calculation, we obtain the fourth column.  Thus, the  fifth column follows  from equation (\ref{eq: invariant_discriminant}). For the last column, we observe that $F_{f_1}(0,0,1)=3$, $F_{f_{2'}}(2,0,1)=4c+3$, $F_{f_{2''}}(0,0,1)=1$, $F_{f_3}(1,1,0)=c'$, $F_{f_4}(1,0,1)=c+4$ and $F_{f_6}(1,0,0)=c$.

    By Brandt\cite{brandt1951zahlentheorie}, \cite{brandt1952mass}, we write the assigned characters of the $f_i$ as follows: $X(f_i)=\{\chi_{-4}\},$ for $i=1, 4 \textrm{ and }6$, $X(f_3)=\{\chi_{-4},\chi_8\}$, and $X(f_2)=\{\chi_{-4},\chi_3\}$ if $3\mid c,$ otherwise, $X(f_2)=\{\chi_{-4}\}$. Since the $f_i$ clearly represent a number $n \equiv1\Mod{4},$ we have that $\chi_{-4}(f_i)=1.$ Moreover, since $f_3(0,0,1)=c'\equiv1\Mod{8},$ it follows that $\chi_8(f_3)=1.$ In the case that $3\mid c$, we see that $\chi_3(f_2) = \chi_3(4) = 1$. Thus, equation (\ref{chi(f)=1}) holds for all forms $f_i.$
    
    Observe  that $\gcd(3,2I_2(f_1))=\gcd(3,8-6c)=1.$ Since $4-3c\equiv1\Mod{3},$ we have that $\left(\frac{4-3c}{3}\right)=1.$  Observe also  that $\gcd(c,8-8c)=1.$ Moreover, we see that $4-4c\equiv 2^2\Mod{c},$ and thus $\left(\frac{4-4c}{c}\right)=1.$ Next, note that  $\gcd(c',4-2c')=1.$ We know that  $\left(\frac{2}{c'}\right)=1,$ for all $c'\equiv1\Mod{8},$ and thus $\left(\frac{2-c'}{c'}\right)=1.$  We lastly check that  $\gcd(c+4,-8c)=1.$ Then we see that $-4c\equiv -4c+c(c+4)\equiv c^2\Mod{c+4},$ which proves that  $\left(\frac{-4c}{c+4}\right)=1.$ We are thus left with the case that $f_2$ to discuss (\ref{I_2(f)=1}). For this, first suppose that $3\nmid c$, and so $I_2(f_2) = -3c$.  We see that $4c+3$  is relatively prime to $2I_2(f_2) = -6c$ since $3\nmid c$. Then we see that $-3c \equiv  -3c+c(4c+3)\equiv (2c)^2\Mod{4c+3},$ which proves that  $\left(\frac{-3c}{4c+3}\right)=1$. We next suppose that $3 \mid c$, and now we have that  $I_2(f_2) = -c/3$. Since $F_{f_2}$ represents 1 in this case,  equation (\ref{I_2(f)=1}) holds trivially. Therefore, equation (\ref{I_2(f)=1}) holds for all the forms $f_i$'s.
    
    We thus conclude that the $f_i$ are equivalent to a refined Humbert invariant by Theorems 1 and 20 of \cite{refhum}. Moreover, they don't represent 1 since they are reduced and their minimum value is their leading coefficient, we obtain from equation (\ref{thetaisirreducible}) that $f_i\sim q_C,$ for some genus 2 curve $C,$ for the $f_i$'s, so the assertion follows.
\end{proof}

We are now ready to start to list the ternary forms $q_C$ based on the information $\Aut(C)$, i.e., to prove the main results. 

Recall  that there is only one genus $2$ curve $C$ up to isomorphism such that $\Aut(C)\simeq C_3\rtimes D_4,$ and $C$ has a model $y^2=x^6-1.$ Also, recall that there is a unique genus 2 curve $C'$ up to isomorphism such that $\Aut(C')\simeq \GL_2(3),$ and $C'$ has a model $y^2=x(x^4-1).$ These two results follow from \cite{shaska2004elliptic}, on pages 711 and 712, and Theorem 2.

We first find the  ternary forms $q_{C}$ and $q_{C'}$ such that $\Aut(C)\simeq \GL_2(3)$ and $\Aut(C')\simeq C_3\rtimes D_4.$

\begin{prop} 
\label{autgl3}
Let $q$ be a ternary form. Then there exists a genus $2$ curve $C$ with
  $\Aut(C)\simeq \GL_2(3)$ such that $q \sim q_C$ if and only if $q_C\sim q_{1,4}.$ Also, there exists a genus $2$ curve $C'$ with  $\Aut(C')\simeq C_3\rtimes D_4$  if and only if $q_{C'}\sim q_{2,4}.$
\end{prop}

\begin{proof}
   We prove both assertions together since a similar argument works for both.

($\Leftarrow$) Note first from Proposition \ref{allformsarehumbert} that  $q_{1,4}\sim q_C$ and $q_{2,4}\sim q_{C'}$ for some genus 2 curves $C,C'.$ In addition, we know that $r_4(q_C)=r_4(q_{1,4})=12$ and $r_4(q_{C'})=r_4(q_{2,4})=8$ by Lemma \ref{r4(q)}. Thus, this direction of the assertions follows by Table \ref{Tab:autC}.
 
 ($\Rightarrow$)  While $\Aut(C)\simeq \GL_2(3),$ there is a unique genus 2 curve $C$  as was stated above. Similarly,  when $\Aut(C')\simeq C_3\rtimes D_4,$ there is a unique genus 2 curve $C'$. Therefore, this direction follows from the previous part for both assertions. 
\end{proof}

We are ready to prove more interesting parts of the assertions in Theorem \ref{firstmainresultPART2}.






\begin{prop} \label{classificationD6}
Let $q$ be a ternary form. Then there exists a genus $2$ curve $C$ with $\Aut(C)\simeq D_6$ such that $q\sim q_C$ if and only if there exists $c>4,$ with $c\equiv0,1\Mod{4}$  such that $q\sim q_{1,c}$ or $q\sim q_{2,c}.$
\end{prop}

\begin{proof}
    ($\Rightarrow$) If $q$ is a primitive form, this part directly follows from Proposition 28 of \cite{cas}. We henceforth suppose that $q$ is imprimitive, and so $\cont(q)=4$ by equation (\ref{content1or4}). Thus, we see that $q$ is an integral quadratic form in the sense of Dickson. By replacing $q$ by an equivalent form, we may assume that $q$ is an Eisenstein reduced form; cf.\ Theorem 103 of \cite{dicksonsbook}. Let us put $q=[a,b,c,2r,2s,2t].$ Note that $a$ is the minimum of the form $q$ and $b$ is the second minimum of the form $q$; cf.\ Theorem 101 of \cite{dicksonsbook}. We know from Table \ref{Tab:autC} that $a(q)=r_4(q)=6.$ Thus, since $\cont(q)=4$ and $r_4(q)>0,$ it follows that $a=4.$ Since $b$ is the second minimum of $q,$ we also see that \begin{equation}\label{r_4(q)=2}
    r_4(q)=2 \Leftrightarrow a=4 \textrm{ and } b\neq4.
\end{equation} Therefore, it follows that $b=4.$ Moreover, since $a(q)=6,$ we have from equations (\ref{athetadividesAut}) and (\ref{Autf}) that $u:=|\Aut^+(q)|\in\{6,12,24\}.$  If $u=24,$ then we would have that $q=q_{4,4,4}$ by Corollary \ref{Aut6_12_24forimprimitiveform} and Lemma \ref{r4(q)} (since $r_4(q)=6$). However, since $\frac{1}{2}q_{4,4,4}$ is not improperly primitive (see \cite[p.~7]{dicksonsbook} for the definition of this concept), $q_{4,4,4}$ cannot be a refined Humbert invariant by Theorem 2 of \cite{Kir}, so this case is not possible. If $u=12,$ then we obtain from Corollary \ref{Aut6_12_24forimprimitiveform} and Lemma \ref{r4(q)} again that there exists $c>4,\ 4\mid c,$ such that $q=q_{2,c}.$

We henceforth suppose that $u=6,$ and so by Corollary \ref{Aut6_12_24forimprimitiveform} it follows that there exists 
$c>4,\ 4\mid c,$ such that $q=q_{1,c},$ so this part follows.

($\Leftarrow$) Let $c>4.$ If $4\mid c$ or $c\equiv1\Mod{4},$ then we see from Propositions \ref{allformsarehumbert} and  \ref{all_primitiveforms_humbert}  that 
there exists a genus 2 curve $C$ such that $q\sim q_C.$ By Lemma \ref{r4(q)}, we have that $r_4(q_{1,c})=6$ for $4\mid c>4$ and that $r_4(q_{2,c})=6,$ for $c>4.$ Also, since $q_{1,c}$ is reduced for $c\geq4$ as was discussed above, we can see from Proposition \ref{prop: Corollary_30} that $a(q_{1,c})=6.$ Hence, we have that $\Aut(q_C)\simeq D_6$ by Table \ref{Tab:autC}, and this part follows.
\end{proof}

We now introduce a notation, which will be used below. Let  $Q_c=\{q_{4,4,c}, q_{6,c}, q_{3,c}\},$  if $c\equiv1\Mod{8},$ and let $Q_c=\{q_{4,4,c}, q_{6,c}\},$ if $c\equiv5\Mod{8},$ and let  $Q_c=\{q_{3,c}, q_{6,c}\},$  if $c\equiv0\Mod{4},$ with $c>4,$ and let $Q_4=\{q_{1,4}, q_{2,4}\}.$

\begin{prop} 
\label{classificationD4} Let $q$ be a ternary form. Then there exists a genus $2$ curve $C$ with $\Aut(C)\simeq D_4$ such that $q\sim q_C$ if and only if there exists $c>4,\ c\equiv0, 1\Mod{4}$ such that $q$ is equivalent to a form in the set $Q_c.$ 

\end{prop}

\begin{proof}  We first claim that \begin{equation} \label{proof_claim_classification_D_4}
    q\sim q_C, \textrm{ with } \Aut(C)\simeq D_4 \Rightarrow u:=|\Aut^+(q)|=4 \textrm{ or }8.
\end{equation}
To prove the claim, assume that $q\sim q_C,$ with $\Aut(C)\simeq D_4.$ In this case,  we know from Table 1 that $a(q)=r_4(q)=4.$ By equation (\ref{content1or4}), we have that $\cont(q)=1$ or $4.$ In the latter case, it is clear that it is a quadratic form in the sense of Dickson, and in the former case, it is also such a form by Proposition \ref{prop: Corollary_30}. Hence, we may assume that $q$ is an Eisenstein reduced form as in the proof of Proposition \ref{classificationD6}. Since $r_4(q)=4,$ we can write that $q=[4,4,c,2r,2s,2t]$ (cf.\ equation (\ref{r_4(q)=2})). Also, since $a(q)=4,$ we have that $u:=|\Aut^+(q)|\in\{4,8,12,24\}$ by equations (\ref{athetadividesAut}) and (\ref{Autf}). Note that the conditions in equation (12) hold for $q$.

We now see that when $q$ is imprimitive, we have that $u\neq 12$ and $u\neq 24$ by Corollary \ref{Aut6_12_24forimprimitiveform} and Lemma \ref{r4(q)}. We henceforth suppose that $q$ is primitive. We first note that $c$ is odd in this case. This implies that $u\neq 24$ by Theorem \ref{propaut6_12_24}. Next, if $u=12,$ then we would obtain that $t=-2$ and $r=s=0$ by the same theorem, and this implies from Proposition \ref{prop: Corollary_30} that $a(q)=6,$ which is not possible. So, equation (\ref{proof_claim_classification_D_4}) follows.

($\Rightarrow$) We first suppose that $q$ is imprimitive. By what was discussed above, we have that  $a(q)=r_4(q)=4,$ and so we may write  $q=[4,4,c,2r,2s,2t].$ Moreover, we have from equation (\ref{proof_claim_classification_D_4}) that $u=8\textrm{ or }4.$ Suppose that $u=8.$ Since $q$ is a positive reduced form and $\cont(q)=4,$ we can apply Theorem  \ref{LemAut4_8}, and we obtain from this result that $q= q_{3,c}$ or $q= q_{4,4,c},$ where $4\mid c>4.$ But, since $\frac{1}{2}q_{4,4,c}$ is not improperly primitive, it cannot be equivalent to a refined Humbert invariant by Theorem 2 of \cite{Kir}, and so this part follows when $u=8.$ 

Suppose next that $u=4.$ Similarly, we can apply Theorem \ref{LemAut4_8}(ii). Since $q(0,1,0)=4,$ we can immediately eliminate the forms having the condition that $b>4$ in this result, and we obtain precisely one form (whose content is 4) $q_{6,c},$ with $4\mid c>4$ (when $b=4$ in the notation of $q_3,$ we have $q_3=q_{6,c}$). Thus this part follows in this case.

We next assume that $q$ is primitive.  By what was discussed above, we have that $a(q)=r_4(q)=4$ and that we may write  $q=[4,4,c,2r,2s,2t],$ where $c$ is odd. Moreover, we have from equation (\ref{proof_claim_classification_D_4}) that $u=8\textrm{ or }4.$ Since $q\sim q_C,$ we have that $q\equiv0,1\Mod{4}$ by equation (\ref{0,1MOD4}), and thus Theorem \ref{LemAut4_8} can be applied.

Suppose first that $u=8.$ By Theorem \ref{LemAut4_8}(i), we have that $q=q_{3,c},$  or $q= q_{4,4,c},$ where $c\equiv1\Mod{4},\ c>4$ since the other forms in this result cannot be primitive in our case. Next, assume that $u=4.$ By Theorem \ref{LemAut4_8}(ii), we see by a similar reason which was discussed in the previous part that $q= q_{6,c},$ where $c\equiv1\Mod{4},\ c>4.$ Hence, this part follows.

($\Leftarrow$) For $c>4,\ 4\mid c,$ we see from Proposition \ref{allformsarehumbert} that 
there exists a genus 2 curve $C$ such that $q\sim q_C.$ For $c, c'>1,$ with $c\equiv1\Mod{4}$ and $c'\equiv1\Mod{8},$ it follows from Proposition \ref{all_primitiveforms_humbert} that there exists genus 2 curve $C$ such that $q\sim q_C.$  We know from Lemma \ref{r4(q)} that $r_4(q_{3,c})=4,$ for $c>4.$ 

We also claim that $r_4(q_{6,c})=4,$ for $c>4.$ To see this, observe that $$q_{6,c}(x,y,z)=4x^2+4y^2+cz^2-4xz=2(x-z)^2+2x^2+4y^2+(c-2)z^2.$$ Since $c-2>4,$ if $(x,y,z)\in R_4(q_{6,c}),$ then $z=0,$ and thus $R_4(q_{6,c})=\{(x,y,0): (x,y)\in R_1(x^2+y^2)\}.$ Thus, the claim follows. Lastly, it is clear that $r_4(q_{4,4,c})=4,$ for $c>4.$ 
Hence, we have that $\Aut(q_C)\simeq D_4$ by Table \ref{Tab:autC}, and the assertions follow.
\end{proof}

\begin{prop} 
\label{classificationC2_C2}
Let $q$ be a positive ternary  form.
\sps

\emph{(i)} If $q$ is imprimitive, then there exists a genus $2$ curve $C$ with $\Aut(C)\simeq C_2\times C_2$ such that $q\sim q_C$ if and only if $q\sim q_{4,c},$ with $4\mid c>4$ or $q\sim q_i,$ for $1\leq i\leq7$  provided that $q_i(0,1,0)\neq4,$ for $i\in\{3,5,6,7\}.$ 
\spm

\emph{(ii)} If $q$ is primitive, and if $q\sim q_C,$ for some curve $C$ of genus $2$ with  $\Aut(C)\simeq C_2\times C_2,$ then  $q$ is equivalent to $[4,b,b,2,4,4],$ with $b>1,\ b\equiv1\Mod{4}$ or $q$ is equivalent to one of the forms listed in  \emph{Theorem \ref{LemAut4_8}(ii) and (iii)} with $b\neq 4.$ 
\end{prop}

\begin{proof} As in the proof of Proposition \ref{classificationD4}, we first claim that
\begin{equation} \label{proof_claim_classification_C_2}
    q\sim q_C, \textrm{ with } \Aut(C)\simeq C_2\times C_2 \Rightarrow u:=|\Aut^+(q)|=2, 4 \textrm{ or }8.
\end{equation}
To prove the claim, assume that $q\sim q_C,$ with  $\Aut(C)\simeq C_2\times C_2.$  As in the proof of equation (\ref{proof_claim_classification_D_4}), we see from Table 1 that $a(q)=r_4(q)=2$ and we may assume that $q$ is an Eisenstein reduced form. Since $r_4(q)=2,$ we can write that $q=[4,b,c,2r,2s,2t],$ with $b\neq 4$ (cf.\ equation (\ref{r_4(q)=2})). Since $q\equiv0,1\Mod{4}$ by equation (\ref{0,1MOD4}), we can apply equation (\ref{Autf}), which shows that $u:=|\Aut^+(q)|\in\{2,4,6,8,12,24\}$ by equation (\ref{athetadividesAut}). 

Assume first that $q$ is primitive, and so $b$ or $c$ is odd. Since $b\neq 4$ and since $b$ or $c$ is odd, we obtain that $u\neq 6, 12\textrm{ and }24$ by Theorem \ref{propaut6_12_24}.
 
We henceforth suppose $q$ is imprimitive. Since $b\neq4,$ we  easily see by Corollary \ref{Aut6_12_24forimprimitiveform} that $u\not\in\{6,24\}.$ In addition, by Corollary \ref{Aut6_12_24forimprimitiveform}, we have that if $u=12,$ then $q=q_{5,c},$ with $4\mid c>4.$ However, since $\frac{1}{2}q_{5,c}$ is not improperly primitive, this is a contradiction by Theorem 2 of \cite{Kir}, and so equation (\ref{proof_claim_classification_C_2}) follows.

We now start to prove the first assertion (i).

($\Rightarrow$)  By what was discussed above, we have that  $a(q)=r_4(q)=2,$ and so we may write  $q=[4,b,c,2r,2s,2t],$ with $b\neq4.$ Moreover, we have from equation (\ref{proof_claim_classification_C_2}) that $u=2, 4\textrm{ or }8.$ As in the proof of Proposition \ref{classificationD4}, since $q$ is reduced positive form and $\cont(q)=4,$ we can apply Theorem  \ref{LemAut4_8}. Thus, we see that $u\neq8$ since $b\neq 4$ and since we know from Theorem 2 of \cite{Kir} that $q$ cannot be a diagonal form.  Thus, we are left with the cases $u=2\textrm{ and } 4$ to analyze.

Suppose first that $u=4.$ By Theorem \ref{LemAut4_8} again, we have that\\ $q\in\{q_1,q_2,q_3,q_4,q'\},$ where  $q'=q_{4,c},$ with $c>4$ since $\cont(q)=4$ and $\frac{1}{2}q$ is improperly primitive.
By equation (\ref{r_4(q)=2}), we can see that $r_4(q')=2$ and $r_4(q_i)=2,$ for $i=1, 2, 4,$ and also we see that $r_4(q_3)=2$ if $q_3(0,1,0)=b\neq 4.$

We next assume that $u=2.$ By  Theorem \ref{LemAut4_8} again, it follows that $q\in\{q_5,q_6,q_7\}.$ We again have that when $b>4,$ $r_4(q_i)=2,$ for $i=5, 6, 7.$ Thus, this part follows.

($\Leftarrow$) First, observe from Proposition \ref{allformsarehumbert} that 
there exists a genus 2 curve $C$ such that $q\sim q_C$ for each case. Moreover,  since $b\neq4$ in each case, we see that $r_4(q_i)=2,$ for $1\leq i\leq 7$ and that $r_4(q_{4,c})=2,$ where $c>4$ by equation (\ref{r_4(q)=2}). Therefore, we obtain from Table \ref{Tab:autC} that $\Aut(C)\simeq C_2\times C_2,$ so this part follows.

We next prove the second assertion (ii). By what was discussed above, we have that $a(q)=r_4(q)=2$ and that we may write  $q=[4,b,c,2r,2s,2t],$ with $b\neq 4.$ Furthermore, we know from equation (\ref{proof_claim_classification_C_2}) that $u=2, 4 \textrm{ or }8.$ Since $q\sim q_C,$  we can apply Theorem \ref{LemAut4_8} as in the proof of Proposition \ref{classificationD4}.

Suppose first that $u=8.$ By Theorem \ref{LemAut4_8}, (since $b\neq4$) we have that  $q=q_{4,b,b},$  or $q=[4,b,b,2,4,4],$ where $b\equiv1\Mod{4},\ b>4.$ But, the first case is not possible since it clearly  represents a number $n\equiv2\Mod{4},$ which contradicts with equation (\ref{0,1MOD4}).

Suppose next that $u=2 \textrm{ or }u=4.$  By Theorem \ref{LemAut4_8}, we have that $q$ is one of the forms listed in Theorem \ref{LemAut4_8}(ii) and (iii) with $b\neq 4.$ Thus, the second assertion follows.
\end{proof}

We are now ready to prove Theorem \ref{firstmainresultPART2}.
\spm

\noindent\textit{Proof of} Theorem \ref{firstmainresultPART2}. As was mentioned in \S\ref{refinedhumbert}, we have that $(A,\theta)\simeq (J_C,C),$ where $\theta=\cl(C).$ It follows that $A\sim E\times E,$ for some CM elliptic curve $E$ by what was mentioned in Remark \ref{remark_C10}. So, the first assertion follows.

We now prove the second assertion. We first prove this assertion in the primitive case, so suppose that  $q_C$ is a primitive ternary form.  For the necessity part, assume that $|\Aut(C)|>2.$ Since $q_C$ is primitive, we see that $u:=|\Aut(C)|=4,8 \textrm{ or }12$ by Proposition \ref{prop: Corollary_30} and equation (\ref{a(q)_Aut(C)}). 
 If $u=12,$ then this part follows by Proposition \ref{classificationD6} and Table \ref{Tab:autC}. If $u=4\textrm{ or }8,$ this part follows respectively by Propositions \ref{classificationC2_C2}(ii) and \ref{classificationD4}. The sufficiency part also follows from those results.

We henceforth suppose that $q_C$ is imprimitive. We observe that if $|\Aut(C)|>2,$  then we have that $u:=|\Aut(C)|=4, 8, 12, 24 \textrm{ or }48;$ cf.\ Table \ref{Tab:autC} and Remark \ref{remark_C10}. Thus the necessity and sufficiency parts  follow from Propositions  \ref{autgl3}, \ref{classificationD6}, \ref{classificationD4}(i) and \ref{classificationC2_C2}(i). The last assertion directly follows from Propositions \ref{allformsarehumbert} and \ref{all_primitiveforms_humbert}. \newline

Moreover, from the results above we can deduce the following useful table of the automorphism groups $\Aut(C)$ for all possible imprimitive ternary forms $q_C.$ It also provides the explicit information for the formulas $a(q_C)$ and $k(q_C).$ The second column contains all possible imprimitive ternary forms of reduced forms $q_C$ which was listed in Theorem \ref{firstmainresultPART2}.  Indeed, all forms in the first seven rows are reduced by Lemma \ref{reducedforms_q7}, and the rest of those forms are reduced by the discussion in \S\ref{ternaryquadraticforms}. The third column follows from Theorems \ref{propaut6_12_24} and \ref{LemAut4_8}. Clearly, the fourth column follows from Table \ref{Tab:autC}, and the last column follows from the third and fourth columns.

\begin{table}[H]
\centering 
  \begin{threeparttable}
    \caption{Imprimitive ternary forms $q_C$ corresponding to $\Aut(C)$ when $a(q_C)\neq1$.}{\label{Tab:k(C)}}
     \begin{tabular}{|c|c|c|c|c|} 
 \hline
  $\Aut(C)$ &   $q_C$ & $|\Aut^+(q_C)|$ & $a(q_C)$ & $k(q_C)$\\ [0.5ex] \hline
  
  $C_2\times C_2$ & $[4,b,c,4,4,4]$ with $b\neq c$ &2&2&1 \\
  & $[4,b,c,2r,0,-4]$ with $r<0$ &2&2&1\\
  & $[4,b,c,2r,-4,0]$ with $r<0$&2&2&1 \\
    &  $[4,b,c,0,0,-4]$&4&2&2 \\
    & $[4,b,b,2r,4,4]$ with $r>0$&4&2&2\\
    &$[4,b,c,0,-4,0]$ with $b\neq c$&4&2&2\\
    &$[4,b,c,-b,-4,0]$ with $\ b\neq c$&4&2&2\\
    &$[4,c,c,-4,0,0]$ &4&2&2\\\hline
   $D_4$  &$[4,4,c,0,-4,0]$ &4 &4&1 \\
   &$[4,4,c,-4,-4,0]$ &8 &4&2\\   \hline
    $D_6$ &  $[4,4,c,4,4,4]$ &6&6&1\\
    &$[4,4,c,0,0-4]$ &12&6&2\\ \hline 
     $C_3\rtimes D_4$ & $[4,4,4,0,0-4]$ &12&12 &1\\ \hline
      $\GL_2(3)$ &  $[4,4,4,4,4,4]$ & 24 &24&1\\ \hline
\end{tabular}

    \begin{tablenotes}
      \small
      \item \textbf{Notation:} In the table, we suppose that $|2r|\leq b,$\ $4\mid b, c, 2r$ and $4<b\leq c.$ In addition, we suppose that if $r<0,$ then $b\neq -2r.$ Also, $k(q_C)=|\Aut^+(q_C)|/a(q_C)$ as was mentioned in the introduction.
    \end{tablenotes}
  \end{threeparttable}
\end{table}

Like Table \ref{Tab:k(C)}, we have the following table for all primitive ternary forms $q_C$ with $a(q_C)>2$. Similarly, the second column contains all possible primitive reduced ternary forms $q_C$  which was listed in Theorem \ref{firstmainresultPART2}.  The third column follows from Theorems \ref{propaut6_12_24} and \ref{LemAut4_8}. The fourth column follows by Table \ref{Tab:autC}, and the last column follows from the third and fourth columns. Note that the table does not contain the forms $q_C$ with $a(q_C) = 2$ even though those forms were discussed in Proposition \ref{classificationC2_C2}(ii). The reason why it was not added to the table is that we don't have the converse of Proposition \ref{classificationC2_C2}(ii).  To give an explicit list for which the forms listed in Proposition \ref{classificationC2_C2}(ii) are equivalent to a refined Humbert invariant seems hard. However, for a given exact form in this list, it is easy by using the main result in \cite{refhum}.

\begin{table}[H]\centering 
  \begin{threeparttable}
    \caption{Primitive ternary forms $q_C$ corresponding to $\Aut(C)$ when $a(q_C)>2$.}{\label{Table: primtiive_case}}
     \begin{tabular}{|c|c|c|c|c|} 
 \hline
  $\Aut(C)$ &   $q_C$ & $|\Aut^+(q_C)|$ & $a(q_C)$ & $k(q_C)$\\ [0.5ex] \hline
  $D_4$ &$[4,4,c,0,-4,0]$ &4 &4&1\\ 
    &$[4,4,c,0,0,0]$ &8 &4&2 \\
   &$[4,4,c',-4,-4,0]$ &8 &4&2\\  
     \hline
    $D_6$ &  $[4,4,c,4,4,4]$ &6&6&1\\
    &$[4,4,c,0,0-4]$ &12&6&2\\ \hline 
\end{tabular}

    \begin{tablenotes}
      \small
      \item \textbf{Notation:} In the table, we suppose that $4 < c, c'$, \text{ and } $c\equiv 1\Mod{4}$ and $c'\equiv 1\Mod{8}$.
    \end{tablenotes}
  \end{threeparttable}
\end{table}

Note that Table \ref{Tab:k(C)} immediately implies that $k(q_C)\leq 2,$ for imprimitive ternary forms $q_C$ with $a(q_C)\neq 1.$ Similarly, it follows from Table \ref{Table: primtiive_case} that $k(q_C)\leq 2,$ for primitive ternary forms $q_C$ with $a(q_C)> 2$.  This is also true for primitive ternary forms $q_C$  with an exact exceptional case when $a(q_C) = 2$ as follows.

\begin{prop}
\label{theoremk(q)}
Suppose that $q_C$ is a ternary form. If $a(q_C)\neq 1$ and if $q_C$ is not equivalent to a form $q:=[4,b,b,2,4,4],$ for some $b\equiv1\Mod{4},$ then $k(q_C)\leq2.$ Moreover, if $q_C\sim q,$ then $k(q_C)=4.$ 
\end{prop} 

\begin{proof}

The first assertion follows from Table \ref{Tab:k(C)} when $q_C$ is imprimitive (cf.\ Table \ref{Tab:autC} for all possible cases of $a(q_C).$) It also follows from Table \ref{Table: primtiive_case} when $q_C$ is primitive with $a(q_C)\neq 2$. Next assume that $q_C$ is a primitive ternary form with  $a(q_C) = 2$  which is not equivalent to   $q$, and so 
   $\Aut(C)\simeq C_2\times C_2$ by Table \ref{Tab:autC}. Then we see that $u=2, 4\textrm{ or }8$ by equation (\ref{proof_claim_classification_C_2}). Moreover, by the proof of Proposition \ref{classificationC2_C2}(ii), we see that $u=8$ only if $q_C\sim q.$  Hence, both assertions follow.
\end{proof}

\begin{remark}
\label{remarka(q)=1}

We can  discuss the case that $a(q_C)=1,$ for a ternary form $q_C.$ By Remark \ref{remark_C10} and Table \ref{Tab:autC}, we know that $\Aut(C)\simeq C_2$ in this case. As was discussed above many times, we may assume that $q_C$ is an Eisenstein reduced form. Since $r_4(q_C)=0$ by Table \ref{Tab:autC}, we have  that $q_C(1,0,0)\neq 4.$ Conversely, if a ternary form $q$ is equivalent to a refined Humbert invariant $q_C,$ for some genus 2 curve $C,$ we again assume that $q$ is reduced, and if $q(1,0,0)\neq4,$ then $r_4(q_C)=0,$ and thus $\Aut(C)\simeq C_2.$

We also observe that Proposition \ref{theoremk(q)} does not hold in the case that $a(q_C)=1.$ 
To this end, consider a reduced positive primitive form $q=[9,16,16,-16,0,0].$ Since $q\equiv0,1\Mod{4}$ and since $q(1,1,1)=5^2,$ which is relatively prime to the discriminant $\disc(q)=-6912,$ we conclude from Theorem 1 of \cite{refhum} that $q$ is equivalent to a refined Humbert invariant. Moreover, since $q$ does not represent 1, it follows from statement (\ref{thetaisirreducible}) that $q\sim q_C,$ for some genus 2 curve $C.$ Since we see that $|\Aut^+(q_C)|=12$ by Theorem \ref{propaut6_12_24}, it follows that $k(q_C)=12.$
\end{remark}

\section{Applications to the intersections} \label{generalizedHumbert}

Kani\cite{MJ} introduced the concept of a \textit{generalized Humbert
scheme} $H(q)$  associated with a given quadratic form $q$ to understand the algebraic
nature of the set of isomorphism classes of the curves lying on abelian product
surfaces. This set is defined by using the refined Humbert invariant as follows: Given any integral positive definite quadratic form $q$ in $r$ variables, let $$H(q)=\{\langle A,\theta\rangle  \in \mathcal{A}_2(K)\, |\,  q_{(A,\theta)} \text{ primitively represents } q\},$$ where $\mathcal{A}_2(K)$ is the set of isomorphism classes $\langle A,\theta\rangle$ of principally polarized abelian surfaces $(A,\theta).$ In other words, $\langle A,\theta\rangle\in H(q)$ if and only if there exists an injective homomorphism $f:\mathbb{Z}^r\xhookrightarrow{} \NS(A,\theta)$ with $\NS(A,\theta)/f(\mathbb{Z}^r)$ torsion-free such that $q=q_{(A,\theta)}\circ f.$ The set $H(q)$ is called a \textit{generalized Humbert scheme.} 

Notice that Dickson  has a slightly different definition for the concept of primitively representness; cf.\ \cite[p.~25]{dicksonsbook}. He also uses the notion of \textit{properly} instead of primitively. By the invariant factor theorem, we see that both definitions are equivalent, and so we can use the results from \cite{dicksonsbook} related to this notion. We use the symbol $f\rightarrow q$ when $f$ primitively (or properly) represents $q.$

Recall from \cite[p.~25]{MJ} that if $q(x)=Nx^2,$ i.e., $r=1,$ then $H(q)=H_N$ is the \textit{classical Humbert surface of invariant $N$}.

By using Tables \ref{Tab:k(C)} and \ref{Table: primtiive_case}, we can get  interesting results related to the intersections of the \textit{generalized Humbert schemes} $H(q)$.

It may be useful to mention that  Milio and  Robert\cite{damien}\footnote{Note that this paper is actually published. However, the authors removed this section for length reasons, and they declared that “they should probably make a separate article about this”; cf.\ \cite[p.~99]{damien2}.} recently found a relation between the denominators of the Hilbert modular polynomials and the generalized Humbert schemes $H(q)$, for a binary form $q$; cf.\ \S5.4 of \cite{damien}.

As in the ternary quadratic forms, we  simply use the abbreviation $q=[a,b,c]$ to denote an integral binary quadratic form $q(x,y)=ax^2+bxy+cy^2.$

We know that for a binary quadratic form $q$ representing a square the nonempty $H(q)$ is a one-dimensional subscheme of $\mathcal{A}_2(K);$ cf.\ Theorem 7.1 of \cite{SubcoversofCurves}. We also know that if $q_1$ and $q_2$ are not equivalent binary quadratic forms such that $H(q_i)\neq\varnothing,$ for $i=1,2,$ then the intersection $H(q_1)\cap H(q_2)$ consists of finitely many \textit{CM points}; cf.\ Lemma 10.7 of \cite{SubcoversofCurves}. 
Hence, we have that $q_{(A,\theta)}$ is a ternary form for any principally polarized abelian surface $(A,\theta)$ in this intersection; cf.\ Section \ref{ternaryquadraticforms}.  

We now provide a proof of Theorem \ref{[4,4,4]}.
\spm

\noindent\textit{Proof of} Theorem \ref{[4,4,4]}. Observe that we have from (\ref{eq: statement_thm_4a_D6}) that  $q_C$ primitively represents the binary quadratic form $q:=[4,4,4]$ if and only if  $\Aut(C)$ contains a subgroup isomorphic to the dihedral group $D_6.$

($\Leftarrow$) Obviously, $q$ primitively represents itself. Also, by what was observed above, this part easily follows by considering $q_{1,c}(1,1,0)=q$ and\\ $q_{2,c}(1,-1,0)=q,$ for any $c$ as in the hypothesis.

($\Rightarrow$) If $q_C$ is a binary form, then we are done by what was observed above. We henceforth assume that $q_C$ is a ternary form.  As in the proof of Lemma \ref{r4(q)}(i), we know that $r_4(q)=6,$ and thus $r_4(q_C)\geq6$ since $q_C\rightarrow q.$ Thus, we have three cases for the value of $r_4(q_C)$ as $6, 8\textrm{ and }12$ by Table \ref{Tab:autC}. If $r_4(q_C)=12,$ or equivalently  $\Aut(C)\simeq \GL_2(3)$ by Table \ref{Tab:autC}, then we have from Proposition \ref{autgl3} that $q_C\sim q_{1,4}.$ In a similar way, if $r_4(q_C)=8,$ we get from the same result that $q_C\sim q_{2,4}.$

 We henceforth suppose that $r_4(q_C)=6.$ Contrary to the previous cases, $q_C$ may be primitive in this case. We first assume that $q_C$ is a primitive form.  By Proposition \ref{classificationD6}, there exists $c\equiv 1\Mod{4}, c>1,$ such that $q_C\sim q_{1,c}$ or $q_{2,c}.$
 
   We next assume that $q_C$ is imprimitive. Since $\Aut(C)\simeq D_6$ in this case (cf.\ Table \ref{Tab:autC}), it follows from Proposition \ref{classificationD6} that there exists $c>4,\ 4\mid c$ such that $q_C\sim q_{1,c}\textrm{ or }q_{2,c}.$ Thus, the first assertion follows. The second assertion directly follows from Propositions \ref{allformsarehumbert} and \ref{all_primitiveforms_humbert}.

To obtain some useful results from Theorem \ref{[4,4,4]} we need the following arithmetic fact, which is interesting and useful itself.

\begin{prop}
\label{prop: upper_bound}
    Let $q$ be a positive definite binary quadratic form such that $q$ does not represent $n$. If a positive definite ternary quadratic form $f$ primitively represents $q$ and $f$ represents $n$, then
    $$
    |\disc(f)| \;\leq\; n|\disc(q)|.
    $$
\end{prop}

\begin{proof}

    Let us put $q=[a, t, b]$.  Since $f$ primitively represents $q$, there is $f'\sim f$ such that $f'(x,y,0) = q(x,y)$; cf.\ \cite[p.~29]{dicksonsbook}, and so, let us put $f'=[a, b, c, r, s, t]$.

By completing the square (Hermite's method), we see that $4 a f'(x, y, z)=(2 a x+ty +s z)^{2}+g(y, z)$, for some positive binary form $g(y, z)$. We explicitly calculate $g(y, z):$
$$
g(y, z)=\left(4 a b-t^{2}\right) y^{2}+(4 a r-2 s t) z y+\left(4 a c-s^{2}\right) z^{2} .
$$
We see that $\disc(g)=(4 a r-2 s t)^{2}-4\left(4 a b-t^{2}\right)\left(4 a c-s^{2}\right)
=16 a(c \disc(q)+q(-r, s))$.

Since $f$ represents $n$, $f'$ represents $n$,  but $q$ does not represent this number, we see that $f'\left(x_{0}, y_{0}, z_{0}\right)=n$, for some $\left(x_{0}, y_{0}, z_{0}\right) \in \mathbb{Z}^{3}$ with $z_{0} \neq 0$. If we apply equation (2.4) of \cite{davidcox} to $g\left(y_{0}, z_{0}\right)$, then we have
$$
\begin{aligned}
g\left(y_{0}, z_{0}\right) & =\frac{\left(2\left(4 a b-t^{2}\right) y_{0}+(4ar-2 s t) z_{0}\right)^{2}-\disc(g) z_{0}^{2}}{4\left(4 a b-t^{2}\right)}\\ &=\frac{\left(2|\disc(q)| y_{0}+\left(4 a r-2 s t) z_{0}\right)^{2}+|\disc(g)| z_{0}^{2}\right.}{4|\disc(q)|} \\
& \geq \frac{|\disc(g)| z_{0}^{2}}{4|\disc(q)|} \geq \frac{|\disc(g)|}{4|\disc(q)|}.
\end{aligned}
$$
Therefore, we see that
$$
4an \;=\; \left(2 a x_{0}+t y_{0}+s z_{0}\right)^{2}+g\left(y_{0}, z_{0}\right) \;\geq\; \frac{|\disc(g)|}{4|\disc(q)|}.
$$
So, we have $4 a n \geq -\frac{16 a(c \disc(q)+q(-r, s))}{4|\disc(q)|}=\frac{16 a(c|\disc(q)| - q(-r, s))}{4|\disc(q)|}$. Since $a>0$, we have that
$$
\begin{aligned}
\quad n \;\geq\; \frac{(c|\disc(q)|-q(-r, s))}{|\disc(q)|}, \quad \text{ and so } \quad n|\disc(q)| \;\geq\; c|\disc(q)|-q(-r, s).
\end{aligned}
$$
The assertion follows since $|\disc(f)|= |\disc(f')| = c|\disc(q)| -q(-r, s)$.
\end{proof}

Here is one of the interesting applications of this arithmetic fact.
\begin{cor}
\label{cor: finite_intersection}
    Let $q$ be a binary quadratic form. If $q$ does not primitively represent $m$, then $H(q)\cap H_m$ is finite.
\end{cor}

\begin{proof}
If the intersection is empty, there is nothing to prove. Otherwise, we first see that $q_{(A, \theta)}$ is a ternary quadratic form for any $(A, \theta)$ in the intersection. In fact, if it was binary, then we would have $q_{(A,\theta)} \sim q$, which is not possible since $q_{(A,\theta)}$ primitively represents $m$.  We next observe that $|\disc(q_{(A, \theta)})|$ has an upper bound. To see this, note that since $(A, \theta) \in H(q) \cap H_m$, we have that $q_{(A, \theta)}$ primitively represents $q$ and $m$. Since $m \notin R(q)$,  it follows by Proposition \ref{prop: upper_bound} that $|\disc(q_{(A, \theta)})| \leq m |\disc(q)|$.

    Recall now that there are finitely many positive definite ternary quadratic forms $q$ such that $|\disc(q)|\leq N$, for a fixed $N$ by Theorem 11 of \cite{watson1960integral}. Furthermore, there are finitely many CM-points whose associated quadratic form is equivalent to a fixed $q$ by what was stated above, and so the assertion follows.
\end{proof}

Note that one can provide a completely different (and a geometric) proof for this corollary by using Theorem 10.5 of \cite{SubcoversofCurves} in the case that the binary form $q$ represents a square.
$$ \text{ Let } H^*(q) \;=\; \{\langle A, \theta \rangle\in H(q): q_{(A,\theta)}\sim q_C, \textrm{ for some curve of genus }2 \}.$$

\begin{cor} 
\label{intersectioncorollary}
Let $q=[4,4,4],$ and let $c>1$  with $c\equiv0, 1\Mod{4}$ such that $q$ does not primitively represent $c.$ Then \begin{equation}
\label{intersection444}
 H^*(q) \;\cap\; H_c \;=\;
    \bigcup_{\mathclap{\substack{q_{i,c'}\rightarrow c, \ i=1, 2 \\ 1<c'\leq c + 1,\ c'\equiv 0, 1\Mod{4}  }}}
          \;(H(q_{1,c'})\cup H(q_{2,c'}))\;
\end{equation}

\end{cor}

\begin{proof}
 Since it is clear that $q_{1,c}$ and $q_{2,c}$ primitively represent $q$ and $c$, for any $c\equiv0,1\Mod{4},$ $c>1$, it follows that the left hand side of equation (\ref{intersection444}) contains the right hand side of that equation. 

For the other inclusion, we first see that for any $(A,\theta)\in H^*(q)\cap H_c,$ the form $q_{(A,\theta)}$ is a ternary form as in the proof of Corollary \ref{cor: finite_intersection}. Note that $q_{(A,\theta)}\sim q_C,$ for some curve $C$ of genus 2  by the definition of $H^*(q).$ Since $q_C\rightarrow q,$  it follows from Theorem \ref{[4,4,4]}  that $q_C\sim q_{1,c'}$ or $q_C\sim q_{2,c'},$ where $c'\equiv0,1\Mod{4},$ with $c'>1.$ 

 We next claim that $c'\leq c +1.$ To prove this, recall first from Table \ref{Table: in_proof} that $|\disc(q_{1,c'})| = 16(3c'-4)$ and $|\disc(q_{2,c'})| = 48c'$. Since $|\disc(q)| = 48$, we obtain from Proposition \ref{prop: upper_bound} that $16(3c'-4) \leq 48c$ and $48c' \leq 48c$, and so we obtain that $c'\leq c +1$ as claimed. 
Thus,  the other inclusion holds, which proves the assertion.
\end{proof}

Note that one can use equation (\ref{intersection444}) to find the intersection $H^*([4,4,4])\cap H(q),$ for a given binary quadratic form $q.$


As was done in Theorem \ref{[4,4,4]}, we determine all refined Humbert invariants $q_C \rightarrow [4,0,4].$ 

\begin{thm} \label{[4,0,4]}
Let $C/K$ be a curve of genus $2.$ Then we have that $q_C$ primitively represents the binary quadratic form $q:=[4,0,4]$ if and only if $q_C\sim q,$ or $q_C$ is equivalent to a form in the set $Q_c,$ with $c\equiv0,1\Mod{4},\ c>1.$ In addition, if $q\in Q_c,$ for any  $c\equiv0,1\Mod{4},\ c>1,$ then there is a curve $C/K$ such that $q_C$ is equivalent to $q.$ 
 
\end{thm}

\begin{proof}  We see that the second assertion directly follows from Propositions \ref{allformsarehumbert} and \ref{all_primitiveforms_humbert}. We now prove the first assertion.

($\Leftarrow$) As in Theorem \ref{[4,4,4]}, this part can be shown by using the theory of ternary quadratic forms. But, we also prove this part by using (\ref{eq: statement_thm_4a_D4}) together with the above results. Indeed, we first observe that  $\Aut(C)$ contains a subgroup isomorphic to the dihedral group $D_4$ in our situation.

Firstly, if $q_C\sim q,$ then the observation directly follows from (\ref{eq: statement_thm_4a_D4}). Secondly, if $q_C\sim q_{i,4},$ for $i=1, 2,$ then we have from Proposition \ref{autgl3} that $\Aut(C)\simeq \GL_2(3) \textrm{ or } C_3\rtimes D_4.$ Lastly, if $q_C\sim q_{j,c},$ for $j=3, 6,$ or $q_C\sim q_{4,4,c},$ where $c$ satisfies the conditions in the hypothesis, then we see from Proposition \ref{classificationD4} that $\Aut(C)\simeq D_4.$ Thus, the observation follows since $D_4\leqslant \Aut(C)$ up to isomorphism in all cases by the group theory.  Hence, this part follows from (\ref{eq: statement_thm_4a_D4}).

($\Rightarrow$) if $q_C$ is a binary form, then it directly follows from (\ref{eq: statement_thm_4a_D4}), and so assume that it is a ternary form.  It is easy to see that $r_1(x^2+y^2)=4,$ and thus we have that $r_4(q)=4,$ and so $r_4(q_C)\geq4.$ Thus, we have four cases for the value of $r_4(q_C)$ as $4, 6, 8\textrm{ and }12$ by Table \ref{Tab:autC}. As in the proof of Theorem \ref{[4,4,4]}, if $r_4(q_C)=12$ or $8,$ then  we have from Table \ref{Tab:autC} that  $\Aut(C)\simeq \GL_2(3)$ or $C_3\rtimes D_4$ respectively, and thus it follows from Proposition \ref{autgl3} that $q_C\sim q\in Q_4.$ 

We next observe that $r_4(q_C)=6$ is not a possible case. If it was possible, then we would have from Table \ref{Tab:autC} that $\Aut(C)\simeq D_6,$ which contradicts with (\ref{eq: statement_thm_4a_D6}) since $D_4$ is not a subgroup of $D_6.$

  We henceforth suppose that $r_4(q_C)=4,$ and it follows from Table \ref{Tab:autC} that $\Aut(C)\simeq D_4.$ Then, the assertion directly follows from Proposition \ref{classificationD4}.
 \end{proof}


In particular, it turns out that we have a simple expression of the intersection $H^*([4,0,4]\cap H_c.$ 

\begin{cor} 
\label{intersectioncorollary2}
Let $q = [4,0,4]$, and let $c>1$  with $c\equiv0, 1\Mod{4}$ such that $q$ does not primitively represent $c.$ Then  \begin{equation} \label{intersection_404}
 H^*(q)\cap H_c=
    \bigcup_{\mathclap{\substack{q'\rightarrow c}}}
          H(q'),
\end{equation} where the union is over all  quadratic forms $q'$ which represent $c$ in the set $\cup_{4\leq c'\leq c+1} Q_{c'}.$
\end{cor}

\begin{proof}
 The left hand side of the equation (\ref{intersection_404}) contains the right hand side of that equation by Theorem \ref{[4,0,4]} as in the proof of Corollary \ref{intersectioncorollary}.

For the other inclusion, for any $(A,\theta)\in H^*(q)\cap H_c,$ we see that $q_{(A,\theta)}$ is a ternary form equivalent to $q_C,$ for some curve $C$
of genus 2; (cf.\ the proof of Corollary \ref{cor: finite_intersection}). Also, since $q_C\rightarrow q,$ it follows from Theorem \ref{[4,0,4]} that $q_C\in Q_{c'},$ for some $c'>1$ with $c'\equiv0,1\Mod{4}$ up to equivalence.

We now show that $c +1\geq c'$ by using a similar argument in the proof of Corollary \ref{intersectioncorollary}. Recall first from Table \ref{Table: in_proof} that $|\disc(q_{3,c'})| = 64(c'-2)$, $|\disc(q_{6,c'})| = 64(c'-1)$ and $|\disc(q_{4,4,c'})| = 64c'$. Since $|\disc(q)| = 64$, we obtain from Proposition \ref{prop: upper_bound} that $64(c'-2) \leq 64c$, $64(c'-1)\leq 64c$ and $64c' \leq 64c$ respectively, and so it follows that $c'\leq c +1$ for the last two cases and that $c'\leq c +2$ for the first situation. Since $c, c' \equiv 0,1 \Mod{4}$, we also have that $c'\leq c +1$ in the last situation, and so the claim follows. This completes the proof of the assertion.
\end{proof}

\begin{ex}\label{example_intersection}  
Consider $H^*(q)\cap H_5,$ where $q=[4,0,4].$ Since $Q_5=\{q_{4,4,5}, q_{6,5}\},$  we have from equation (\ref{intersection_404}) that $H^*(q)\cap H_5=H(q_{4,4,5})\cup H(q_{6,5}).$

\end{ex}

We are in a position to verify Corollary \ref{emptyH(q1)H(q2)} stated in the introduction. This result shows that the analogue of Corollary 8.2 of \cite{SubcoversofCurves} (or Corollary 4 of \cite{Kir}) is not correct for the intersection of the generalized Humbert schemes $H(q),$ where $q$ is a binary form.\newline

\noindent\textit{Proof of} Corollary \ref{emptyH(q1)H(q2)}. It is clear that $q_1$ and $q_2$ are positive forms. Therefore, the first assertion follows from  Theorem 8.1 of \cite{SubcoversofCurves} since $q_1\rightarrow 2^2$ and $q_2\rightarrow 3^2,$ and since $q_1\equiv 0\Mod{4}$ and $q_2\equiv (x+y)^2\equiv0,1\Mod{4}.$ Moreover, since $q_2$ is primitive, it follows from Theorem 9.4 of \cite{SubcoversofCurves} that $H(q_2)$ is an irreducible curve. Also, Proposition 51 (or Proposition 57) of \cite{generalizedhumbert}, we have that $H(q_1)$ is an irreducible curve.

Since $q_1$ and $q_2$ are not equivalent,  for any $(A,\theta)\in H(q_1)\cap H(q_2)$ (if exists) we see that $f:=q_{(A,\theta)}$ is a ternary form (cf.\ the proof of Corollary \ref{cor: finite_intersection}).

If $f$ does not represent 1, then $f$ is equivalent to a form in $Q_5$ by Example \ref{example_intersection} since $H(q_2)\subset H_5.$ However,  we claim that neither  $q_{4,4,5}=[4,4,5,0,0,0]$ nor $q_{6,5}=[4,4,5,0,-4,0]$ can primitively represent $q_2.$  To see this, we observe\footnote{ This observation can be achieved by more elementary ways, e.g., while $q_2 \rightarrow 16$, the form $q_{4,4,5}$ cannot represent primitively 16. However, we believe that the approach applied above is more useful for further applications.} that the reciprocal forms $F_1=F_{q_{4,4,5}} = [5,5,4,0,0,0]$ and $F_2=F_{q_{6,5}} = [5,4,4,0,4,0]$ cannot represent any number $n\equiv 3\Mod{4}$; cf.\ Table \ref{Table: in_proof}. Since we can see that $|I_1|:=|I_1(q_{4,4,5})|=|I_1(q_{6,5})|=16,$  (cf.\ Table \ref{Table: in_proof}), it follows that if $q_2$ was represented primitively  by one of them, then we would have from Theorem 27 of \cite{dicksonsbook} (together with the relations given on \cite[p.~316]{brandt1952mass}) that $F_{i}\rightarrow \frac{-\disc(q_2)}{|I_1|}=\frac{432}{16}=27,$ which contradicts with the observation. Hence, we conclude the claim that both forms cannot primitively represent $q_2$. Therefore, we obtain that $f\rightarrow 1.$ 

In this case, we have that $f\sim f_q:=x^2+4q,$ for some binary form $q$ (cf.\ equation (29) of \cite{refhum}). Since $f\rightarrow q_1, q_2,$ we have that $f_q\rightarrow 4, 9,$ and thus $q\rightarrow 1, 2.$ Since if $q\sim q',$ then $f_q\sim f_{q'},$ we may change $q$ by a reduced form $q',$ and assume that $q$ is a reduced form. Since $q\rightarrow 1,2,$ it follows that $q=[1,b,c],$ with $c\leq2$ by Theorem 5.7.6 of \cite{vollmer2007binary}. By the properties of a reduced form for a binary form (cf.\ Watson\cite[p.~14]{watson1960integral}), we have that $b=0 \textrm{ or } 1,$ and thus we have four possible cases for $q.$ However, we immediately eliminate the case that $q=[1,1,1]$  since it cannot represent 2, and so we have three cases:
\sps

\noindent\textbf{Case 1.} Assume that $q=[1,0,1].$
\spm

\noindent Since the reciprocal form of $f_q$ is $F_{f_q}=[4,1,1,0,0,0]$ in this case, it is clear that $F_{f_q}$ cannot represent a number $n\equiv 3\Mod{4},$ which implies that $f_q$ cannot represent $q_2$ by what was discussed above, and so this case is not possible.
\sps

\noindent\textbf{Case 2.} Assume that $q=[1,0,2].$
\spm

\noindent We see that $F_{f_q}=[8,2,1,0,0,0].$ It is clear that this form cannot represent 4. Since $\frac{-\Delta(q_1)}{|I_1(f_q)|}=\frac{64}{16}=4,$ we again obtain from Theorem 27 of \cite{dicksonsbook} that $f_q$ cannot represent $q_1.$ So, this case is not possible.
\sps

\noindent\textbf{Case 3.} Assume that $q=[1,1,2].$
\spm

\noindent In this case, we have that $F_{f_q}=[7,2,1,-1,0,0].$ Note that $F_{f_q}\equiv 2y^2-yz+z^2\equiv (3y+z)^2\Mod{7},$ which implies that $F_{f_q}$ cannot represent a quadratic nonresidue modulo 7. It thus follows that  $F_{f_q}$ cannot represent 27 since it is not a quadratic residue modulo 7. Thus, we have by what was discussed above that $f_q$ cannot represent $q_2$ in this case, so this case is not possible. Therefore, we have proved that the intersection $H(q_1)\cap H(q_2)=\varnothing.$

\begin{ex}
\label{ex: intersections_3_humbert_surfaces}    
     Here, we want to discuss the (corresponding) curves $C$ of genus 2 with $\Aut(C) \not\simeq C_{2} \times C_{2}$ in the intersection of three exact Humbert surfaces ${H}_{4} \cap {H}_{5} \cap {H}_{8}$. By the formula in 5.2 of \cite{SubcoversofCurves}, we have that if $n \neq m$, then 
     $$
     H_{n} \;\cap\; H_{m} \;=\;\bigcup_{q \rightarrow n, m} H(q),
     $$
     where the union is over all positive definite binary quadratic forms $q$ which represent both $n$ and $m$ primitively. By using the \textit{reduction theory} of binary quadratic forms and this formula, we can find that
$$
H_{4} \cap H_{5}=H([1,0,4]) \cup H([4,0,5]) \cup H([4,4,5]) \text {; }
$$
cf. Proposition 57 of \cite{generalizedhumbert}. Thus, it follows that
$$
H_{8} \cap\left(H_{4} \cap H_{5}\right)=H_{8} \cap(H([1,0,4]) \cup H([4,0,5]) \cup H([4,4,5])).
$$
Via the Torelli map (cf.\ the proof of Proposition \ref{p: bijection} below) we can view the moduli space $\mathcal{M}_{2}$ of the curves of genus 2 as a subset of $\mathcal{A}_{2}$. In this regard, one can see that $H^{*}(q)=H(q) \cap \mathcal{M}_{2}$. By the irreducibility criterion (\ref{thetaisirreducible}), we have that 
$$
\mathcal{H} \;:=\; H_{8} \cap H_{4} \cap H_{5} \cap \mathcal{M}_{2} \;=\; H_{8} \cap(H([4,0,5]) \cup H([4,4,5]).
$$
(Note that $\mathcal{M}_{2} \cap H([1,0,4])=\varnothing$). Now, it is easy to see that $f_{1}:=[4,0,5]$ and $f_{2}:=[4,4,5]$ cannot represent 8, and hence the intersection $\mathcal{H}$ is finite by Corollary \ref{cor: finite_intersection}. Moreover, for any curve $C \in \mathcal{H}$, we have that $J_{C}$ is a CM product abelian surface since $q_{C}$ is ternary. By using the upper bound coming from Proposition \ref{prop: upper_bound} (and the classification of the refined Humbert invariant), one can find all refined Humbert invariants $q_{C}$ corresponding to CM points on $\mathcal{H}$. This is possible by using the relation between the representation numbers of the reciprocal form $F_{f}$ and the represented binary forms by a ternary form $f$ as was discussed in the proof of Corollary \ref{emptyH(q1)H(q2)} (cf.\ Theorem 40 of \cite{dicksonsbook}). 

For simplicity, we want to determine all $q_{C}$ such that $C \in \mathcal{H}$ with $\Aut(C) \not\simeq C_{2} \times C_{2}$. For any $C \in \mathcal{H}$, since $q_{C} \rightarrow 5$, we have that $q_{C}$ is primitive by (\ref{content1or4}). Moreover, if $\Aut(C) \not\simeq C_{2}\times C_{2}$, then since $q_{C} \rightarrow 4$, it follows from Table \ref{Table: primtiive_case} that $\Aut(C) \simeq D_{4}$ or $D_{6}$. Thus, $q_{C} \rightarrow [4,0,4]$ or $[4,4,4]$ respectively by (\ref{eq: statement_thm_4a_D4}) and (\ref{eq: statement_thm_4a_D6}). This implies that $C \in H^{*}[4,0,4] \cap H_{5}$ or $C \in H^{*}[4,4,4] \cap H_{5}$ respectively, and so we apply equations (\ref{intersection444}) and (\ref{intersection_404}). Hence, we obtain the possible list: $q_{C} \in\left\{q_{4,4,5}, q_{6,5}, q_{1,5}, q_{2,5}\right\}$ (up to equivalence).

We immediately see that $q_{6,5}(1,0,1)=[4,-4,5] \sim f_{2}$ and $q_{6,5}(-1,-1,0) = 8$, and so if $q_C\sim q_{6,5}$, then $C \in H_{8} \cap H\left(f_{2}\right)$. We also have that $q_{1,5}$ and $q_{2,5}$ cannot represent 8 primitively. (One can see this by using the expressions in (\ref{eq: expression_q_1c}) and (\ref{eq: expression_q_2c}) respectively). We see that $f_{2}(1,-2)=16$, but $q_{4,4,5}$ cannot primitively represent 16, and so $q_{4,4,5}$ cannot represent primitively $f_{2}$. Hence, if $C \in \mathcal{H}$ and $\Aut(C) \not\simeq C_{2} \times C_{2}$, then $q_{C} \sim q_{6,5}$.
\end{ex}

\section{Elliptic subcovers of a curve of genus 2}
\label{s: elliptic_subcovers}

Let $K$ be an algebraically closed field with $\charec(K) = 0$, and let $C/K$ be a curve of genus 2 over $K$. Recall that an \textit{elliptic subcover} is a finite morphism $\phi: C\rightarrow E$ to an elliptic curve $E/K$ which does not factor over a non-trivial isogeny (i.e., maximal covering) of $E$. Its degree is the degree of the morphism. 
There is a considerable literature on elliptic subcovers. For the reader, we refer to the references given in Bruin and Doerksen\cite{Nils} regarding this topic.
We want to show the
relation  of elliptic subcovers to our main results,  let us recall from
Theorem 1 of Kani\cite{ESCI} that for a curve $C$ of genus 2 we have
\begin{equation}
    \label{eq: elliptic_subcover_fact}
    C/K \text{ has an elliptic subcover of degree } n \Leftrightarrow q_C \text{ primitively represents } n^2.
\end{equation}

 For a given group $G$ and a given number $n$,
 $$
 \mathcal{H}(G,n) \;:=\; \{\langle C \rangle :  C/K \text{ has an elliptic subcover of degree } n \text{ with }\Aut(C)\simeq G \}
 $$
denote the set of isomorphism classes of genus 2 curves $C/K$ with $\Aut(C) \simeq G$ such that it has an elliptic subcover of degree $n$. Shaska\cite{Shaska_subcovers} studied the sets  $ \mathcal{H}(D_4,3) \text{ and } \mathcal{H}(D_6,3) $, and found their cardinalities.  Kani\cite{SubcoversofCurves} and Mukamela\cite{orbifolds} discussed this kind of set in terms of the moduli point of view.   Note that we can  formulate  these sets with the intersection $H^*([4,4,4])\cap H_{m^2}$ (and $H^*([4,0,4])\cap H_{m^2}$) by Torelli map because we have:

\begin{prop}
\label{p: bijection}
    The map $ C\mapsto (J_C, \theta_C)$ induces a bijection between $\mathcal{H}(D_6,m)$ and $H^*([4,4,4])\cap H_{m^2}$, for any odd $m$. 
    Therefore, we conclude that the cardinality of these finite sets is equal, i.e.,
   \begin{equation}
   \label{eq: cardinalities_same_H(D_6)_H(4,4,4)}
       |\mathcal{H}(D_6,m)| \;=\; |H^*([4,4,4])\cap H_{m^2}|.
   \end{equation}
\end{prop}

\begin{proof} It is clear that if $\mathcal{H}(D_6,m) = \varnothing$, for some $m$, then $H^*([4,4,4])\cap H_{m^2} = \varnothing$ by what was discussed above. Otherwise, to see that the given rule induces a map, let us take $\langle C \rangle \in \mathcal{H}(D_6,m)$, and so $q_C \rightarrow m^2$ by equation (\ref{eq: elliptic_subcover_fact}). This implies that $\langle J_C, \theta_C\rangle \in H_{m^2}$. Moreover, since $m$ is odd, $q_C$ cannot be equivalent to $[4,4,4]$. Therefore,  $q_C$ is a ternary form that primitively represents $[4,4,4]$ by Theorem \ref{[4,4,4]}, and so $\langle J_C, \theta_C \rangle \in H^*([4,4,4])$ by the definition. Thus, it is easy to see that the given rule induces a map $\Psi: \mathcal{H}(D_6,m) \rightarrow H^*([4,4,4])\cap H_{m^2}$. 
   
   We know that this map is injective by Torelli’s Theorem; cf.\ Theorem 12.1 of \cite{milne1986jacobian}, and so it suffices to show that it is surjective. For this, let us take $\langle A, \theta \rangle \in H^*([4,4,4])\cap H_{m^2}$, and it follows from the definition that $\langle A, \theta \rangle = \langle J_C,\theta_C \rangle$, for some curve $C$ of genus 2. Since $q_C$ primitively represents an odd number $m^2$ (by definition), this implies that $C$ has an elliptic subcover of degree $m$ by equation (\ref{eq: elliptic_subcover_fact}) and that $q_C$ is a primitive ternary form (cf.\ (\ref{content1or4})). In addition, we see that $r_4([4,4,4]) = 6$ (cf.\ the proof of Lemma \ref{r4(q)}(i)), and so $r_4(q_C) \geq 6$ by the assumption that $q_C \rightarrow [4,4,4]$. Hence, it follows from Proposition \ref{prop: Corollary_30} that $r_4(q_C) = 6$, and so $\Aut(C)\simeq D_6$ by Table \ref{Tab:autC}. Thus, it follows that $ \langle J_C,\theta_C\rangle = \Psi(\langle C \rangle)$ lies in the image of $\Psi$, which shows that $\Psi$ is surjective, and so it is bijective as asserted.

   We know that $|H^*([4,4,4])\cap H_{m^2}|$ is finite by Corollary \ref{cor: finite_intersection}, and so, equation (\ref{eq: cardinalities_same_H(D_6)_H(4,4,4)}) follows.
\end{proof}

By the similar arguments in the proof of Proposition \ref{p: bijection},  $\mathcal{H}(D_6,m)$ can be formulated in terms of $H^*([4,4,4])\cap H_{m^2}$, for all $m$ (not necessarily odd) provided $[4,4,4]$ does not represent $m$ primitively. For this, one only has to take care of the fact that $(J_C,\theta_C)$  might be in $H^*([4,4,4])\cap H_{m^2}$ when $\Aut(C) \simeq \GL_2(3)$ or $C_3 \rtimes D_4$.

For any odd $m$, we have from (\ref{intersection444}) that
$$
H^*([4,4,4]) \;\cap\; H_{m^2} \;=\;
    \bigcup_{\mathclap{\substack{q_{i,c'}\rightarrow m^2, \ i=1, 2  }}}
          \;(H(q_{1,c'})\cup H(q_{2,c'})),\;
$$ where $1<c'\leq m^2+1$ and $ c'\equiv 1\Mod{4} $. (Note that $c'$ cannot be even).  By the definition, we see that  the union  $H(q_{1,c'})\cup H(q_{2,c'})$ is disjoint for any $c'$. Therefore, it follows from (\ref{eq: cardinalities_same_H(D_6)_H(4,4,4)}) that 
\begin{equation}
\label{eq: |H(D6,m)|=|H(q)|}
     |\mathcal{H}(D_6,m)| \;=\; \sum_{\substack{q_{i,c'} \rightarrow m^2 }} |H(q_{1,c'})|+| H(q_{2,c'})|, 
\end{equation}
where $1<c'\leq m^2+1$, $ c^{\prime} \equiv 1 \Mod{4}$ and $i = 1,2$. Thus, the cardinality $|\mathcal{H}(D_6,m)|$ is formulated in terms of $|H(q)|$, for certain ternary forms $q$.

\begin{ex}
\label{ex: Shaska's_verified}
We want to illustrate the case that $m = 3$ in Proposition \ref{p: bijection} because the set $\mathcal{H}(D_6,3)$ was explicitly discussed in the main results of \cite{Shaska_subcovers}. For this,  observe first that $q_{i,c}$ cannot represent 9 for an even number $c$, and for $i=1, 2$. Moreover, it is clear that $q_{i,9}(0,0,1) = 9$, for $i=1, 2$ and that $q_{1,5}(-1,-1,1) = 9$ and $q_{2,5}(1,1,1) = 9$. Thus, by Corollary \ref{intersectioncorollary}, we have that
\begin{equation}
\label{eq: H(4,4,4)andH_9}
    H^*([4,4,4])\cap H_9 \;=\; H(q_{1,5})\cup H(q_{1,9})\cup H(q_{2,5}) \cup H(q_{2,9}),
\end{equation}
and it follows by (\ref{eq: |H(D6,m)|=|H(q)|}) that
$$
 |\mathcal{H}(D_6,3)| \;=\; |H(q_{1,5})| \;+\; |H(q_{1,9})| \;+\; |H(q_{2,5})| \;+\; |H(q_{2,9})|.
$$


\end{ex}

By using the properties of the refined Humbert invariant, it is possible to reprove many parts of Shaska's main results; cf.\ Remark \ref{rem: cardinality_of_H(D_6)} and Example \ref{ex: q_1_5} below. Recall that the sets $\mathcal{H}(D_6,3)$ and $\mathcal{H}(D_4,3)$ are determined, and the equations/models for the curves in these sets defined over $\Q$ are given in Theorem 1, Corollary 1 and Table 1 of \cite{Shaska_subcovers}.

\begin{remark}
\label{rem: cardinality_of_H(D_6)}
  
By what was mentioned above, if one finds a formula for the number $|H(q)|$, for a ternary form $q$, then it leads to a formula for the cardinality $|\mathcal{H}(D_6,m)|$. Recently,  the formulas for $|H(q)|$, for any ternary form $q$, have been found and are in preparation by a joint work in preparation with Kani. By using this formula, we obtain that if $h(\Delta)$ denotes the class number of the discriminant $\Delta$, then
$$
|\mathcal{H}(D_6,m)| \;=\; \sum_{\substack{q_{1,c'} \rightarrow m^2 }} h\left(4-3 c^{\prime}\right) + \sum_{\substack{q_{2,c'} \rightarrow m^2 \\ 3\mid c'}} h\left(-3c^{\prime}\right)  + \sum_{\substack{q_{2,c'} \rightarrow m^2 \\  3\nmid c'}} \frac{h\left(-3c^{\prime}\right)}{2},
$$
where $5 \leqslant c^{\prime} \leqslant m^2+1$ and $c^{\prime} \equiv 1 \Mod{4}$. Notice that it easily follows that $|\mathcal{H}(D_6,3)| = 6$ by this formula.
\end{remark}

Observe that one similarly proves the analogue of Proposition \ref{p: bijection} for $D_4$, and shows the bijection between $\mathcal{H}(D_4,m)$ and $H^*([4,0,4])\cap H_{m^2}$, and obtain a formula for the cardinality $|\mathcal{H}(D_4,m)|$ by what was discussed in Remark \ref{rem: cardinality_of_H(D_6)}.

 We  can say more than finding the cardinality $|\mathcal{H}(D_6,3)|$  about the curves $C/K$ in $\mathcal{H}(D_6,3)$ by using the properties of $q_C$. The following examples illustrate two curves in the set $\mathcal{H}(D_6,3)$.

\begin{ex}
\label{ex: q_1_5}
    We discuss the curve $C$ with $q_C\sim q_{1,5}$ (cf.\ Example 6 of \cite{Shaska_subcovers}). By Remark \ref{remark_C10} and equation (\ref{thetaisirreducible}), we know that $J_C \simeq E_1\times E_2$ is a CM product surface. Recall that for two isogenous elliptic curves $E,E',$ we have the \textit{degree} form $q_{E,E'}$ on $\Hom(E,E')$ which is defined by $q_{E,E'}(h)=\deg(h),$ for $h\in \Hom(E,E')$, and we know that it is a positive definite binary quadratic form in the CM case. Thus, $q_{E_1,E_2}$ is a binary form, and by Lemma 30 and equation (37) of \cite{ESCI}, we see that $ \disc(q_C)/16 = \disc(q_{E_1,E_2}) = -11$, and so it easily follows that $q_{E_1,E_2}\sim [1,1,3]$  by the reduction theory of binary quadratic forms. Since it represents $1$,  it follows that $E_1\simeq E_2$. Let us put $F = \End(E_1)\otimes \Q$. By Corollary 42 of \cite{kani2011products} we find that $\disc(q_{E_1, E_2}) = \Delta_F$, where $\Delta_F$ is the discriminant of the quadratic number field $F$ and that $\End(E_1)$ is isomorphic to the maximal order of $F$ (since $-11$ is a fundamental discriminant).
    
     Moreover, since the class number $h(-11) = 1$, we know that the isomorphism class of the elliptic curve $E_1/K$ is uniquely determined by  the endomorphism ring of $E_1/K$; cf.\ equation (55) of \cite{kani2011products}. Moreover, we see from \cite[p.~483]{silverman1994advanced} that the $j$-invariant $j(E_1) = -2^{15}$ and that $E_1/K$ is given by the equation $y^2 + y = x^3 -x^2 - 7x +10$.

We can next find that $|H(q_{1,5})| = 1$ by using the formula mentioned in Remark \ref{rem: cardinality_of_H(D_6)}. However, since we don't provide proof for that formula in this article, we don't want to use it in this example. Instead, we use Hayashida's formula\cite{hayashida1968class}, \S7 (cf.\ the computer calculations of \cite{gelin2019principally} derived from Hayashida's formula). 

 Note that we are in the case that $J_C\simeq E_1\times E_1$, where $\End(E_1)$ is the maximal order of $\Q(\sqrt{-11})$, and thus Hayashida's formula applies, 
 and it implies that there is a unique curve $C'$ of genus 2 on $E_1\times E_1 \simeq J_C$, and so $C\simeq C'$. Thus, we conclude from Table 4 of \cite{gelin2019principally} that   $C$ is given by the equation $ y^2 = 2x^6 + 11x^3 - 22$. Observe that this curve is isomorphic to the curve $C'': y^2 = 11x^6 + 11x^3 - 4$, where $C''$ is the curve which was given in the sixth row of Table 1 of \cite{Shaska_subcovers} because both curves have the same  \textit{Igusa-Clebsch invariants} $(I_2, I_4, I_6, I_{10}) = (180576, 361304064, 31762962874368, 38770566118886080512)$.
\end{ex}

\begin{ex}
    Consider the curve $C$ with $q_c \sim q_{2,5}$ (cf. Example 7 of \cite{Shaska_subcovers}). As in Example \ref{ex: q_1_5}, we see that $J_C \simeq E_1\times E_2$ is a CM abelian product surface, and $-16 \disc\left(q_{E_1, E_2}\right)=\disc\left(q_{2,5}\right)=-16.15$ (cf.\ Table \ref{Table: in_proof}). By Corollary 42 of \cite{kani2011products}, we see that $\End(E_1) \simeq \End(E_2)$ is the maximal order $\mathcal{O}$ of $\Q(\sqrt{-15})$. 
    
    If $q_{E_1, E_2}$ represents 1, then $E_1 \simeq E_2$, but then by Hayashida/Nishi\cite{hayashida1965existence} (or Theorem 2 of \cite{kani2014jacobians}), $E_1 \times E_2$ cannot be a Jacobian of a genus 2 curve, which is a contradiction, and so $q_{E_1, E_2}$ cannot represent 1 . Thus, by the reduction theory, it follows that $q_{E_1, E_2} \sim[2,1,2]$, and so this implies that $J_C\not\simeq E\times E$, for any CM elliptic curve $E$. Since $h(-15)=2$, there are two elliptic curves $E/K$ with $\End(E)=\mathcal{O}$ up to isomorphism over $K$. One can determine the $j$-invariants of these elliptic curves by finding the roots of the Hilbert class polynomial of discriminant $-15$; e.g., see the elliptic curves with LMFDB labels \cite[\href{https://www.lmfdb.org/EllipticCurve/2.2.5.1/81.1/a/1}{Elliptic Curve 81.1.a1}]{lmfdb} and \cite[\href{https://www.lmfdb.org/EllipticCurve/2.2.5.1/81.1/a/3}{Elliptic Curve 81.1.a3}]{lmfdb}.

     By using the formulas in \cite{cas}, we see that there is a unique genus 2 curve on $J_C$ up to isomorphism.

     We observe by equation (\ref{eq: expression_q_2c}) that
     $$
     R_9(q_{2,5}) \;=\; \{\pm(0,1,1), \pm(0,-1,1), \pm(1,0,1), \pm(-1,0,1), \pm(-1,-1,1), \pm(1,1,1) \},
     $$ 
and so there are $12$  primitive triples $(a,b,c)$ such that $q_{2,5}(a,b,c) = 9$, and thus there are 12 elliptic subcovers of degree $3$ of $C$ (up to isomorphism) by Theorem 4.5 of \cite{kani1994elliptic}.
\end{ex}

\bibliographystyle{plainurl}
\bibliography{kir_aut.bib}

\begin{bibdiv}
\begin{biblist}

\bib{accola}{article}{
      author={Accola, R. D.~M.},
      author={Previato, E.},
       title={Covers of tori: genus two},
        date={2006},
        ISSN={0377-9017,1573-0530},
     journal={Lett. Math. Phys.},
      volume={76},
      number={2-3},
       pages={135\ndash 161},
         url={https://doi.org/10.1007/s11005-006-0067-5},
      review={\MR{2235401}},
}

\bib{brandt1952mass}{article}{
      author={Brandt, H.},
       title={\"{U}ber das {M}ass positiver tern\"{a}rer quadratischer {F}ormen},
        date={1952},
     journal={Math. Nachr.},
      volume={6},
       pages={315\ndash 318},
         url={https://doi.org/10.1002/mana.19520060507},
      review={\MR{51268}},
}

\bib{brandt1951zahlentheorie}{article}{
      author={Brandt, H.},
       title={Zur {Z}ahlentheorie der tern\"{a}ren quadratischen {F}ormen},
        date={1952},
     journal={Math. Ann.},
      volume={124},
       pages={334\ndash 342},
         url={https://doi.org/10.1007/BF01343574},
      review={\MR{51269}},
}

\bib{Nils}{article}{
      author={Bruin, N.},
      author={Doerksen, K.},
       title={The arithmetic of genus two curves with {$(4,4)$}-split {J}acobians},
        date={2011},
        ISSN={0008-414X,1496-4279},
     journal={Canad. J. Math.},
      volume={63},
      number={5},
       pages={992\ndash 1024},
         url={https://doi.org/10.4153/CJM-2011-039-3},
      review={\MR{2866068}},
}

\bib{vollmer2007binary}{book}{
      author={Buchmann, J.},
      author={Vollmer, U.},
       title={Binary quadratic forms},
      series={Algorithms and Computation in Mathematics},
   publisher={Springer, Berlin},
        date={2007},
      volume={20},
        ISBN={978-3-540-46367-2; 3-540-46367-4},
        note={An algorithmic approach},
      review={\MR{2300780}},
}

\bib{cardona_quer}{article}{
      author={Cardona, G.},
      author={Quer, J.},
       title={Curves of genus 2 with group of automorphisms isomorphic to {$D_8$} or {$D_{12}$}},
        date={2007},
        ISSN={0002-9947,1088-6850},
     journal={Trans. Amer. Math. Soc.},
      volume={359},
      number={6},
       pages={2831\ndash 2849},
         url={https://doi.org/10.1090/S0002-9947-07-04111-6},
      review={\MR{2286059}},
}

\bib{davidcox}{book}{
      author={Cox, D.~A.},
       title={Primes of the form {$x^2 + ny^2$}},
      series={A Wiley-Interscience Publ.},
   publisher={John Wiley \& Sons, Inc., New York},
        date={(1989)},
         url={https://onlinelibrary.wiley.com/doi/book/10.1002/9781118032756},
      review={\MR{1028322}},
}

\bib{dicksonsbook}{book}{
      author={Dickson, L.},
       title={Studies in number theory},
   publisher={U Chicago Press, Chicago, 1930. Reprinted by Chelsea Publ. Co., New York},
        date={(1957)},
}

\bib{gelin2019principally}{inproceedings}{
      author={G\'{e}lin, A.},
      author={Howe, E.},
      author={Ritzenthaler, C.},
       title={Principally polarized squares of elliptic curves with field of moduli equal to {$\mathbb{Q}$}},
        date={(2019)},
   booktitle={Proc. 13th {A}lgorithmic {N}umber {T}heory {S}ymposium},
      series={Open Book Ser. 2},
       pages={257\ndash 274},
         url={https://msp.org/obs/2019/2-1/p16.xhtml},
      review={\MR{3952016}},
}

\bib{hayashida1968class}{article}{
      author={Hayashida, T.},
       title={A class number associated with the product of an elliptic curve with itself},
        date={1968},
        ISSN={0025-5645,1881-1167},
     journal={J. Math. Soc. Japan},
      volume={20},
       pages={26\ndash 43},
         url={https://doi.org/10.2969/jmsj/02010026},
      review={\MR{233804}},
}

\bib{hayashida1965existence}{article}{
      author={Hayashida, T.},
      author={Nishi, M.},
       title={Existence of curves of genus two on a product of two elliptic curves},
        date={1965},
        ISSN={0025-5645,1881-1167},
     journal={J. Math. Soc. Japan},
      volume={17},
       pages={1\ndash 16},
         url={https://doi.org/10.2969/jmsj/01710001},
      review={\MR{201434}},
}

\bib{Igusa}{article}{
      author={Igusa, J.-I.},
       title={Arithmetic variety of moduli for genus two},
        date={1960},
        ISSN={0003-486X},
     journal={Ann. of Math. (2)},
      volume={72},
       pages={612\ndash 649},
         url={https://doi.org/10.2307/1970233},
      review={\MR{114819}},
}

\bib{kani1994elliptic}{article}{
      author={Kani, E.},
       title={Elliptic curves on abelian surfaces},
        date={1994},
     journal={Manuscripta Math.},
      volume={84},
      number={2},
       pages={199\ndash 223},
         url={https://doi.org/10.1007/BF02567454},
      review={\MR{1285957}},
}

\bib{kani2011products}{article}{
      author={Kani, E.},
       title={Products of {CM} elliptic curves},
        date={2011},
     journal={Collect. Math.},
      volume={62},
      number={3},
       pages={297\ndash 339},
         url={https://doi.org/10.1007/s13348-010-0029-1},
      review={\MR{2825715}},
}

\bib{kani2014jacobians}{article}{
      author={Kani, E.},
       title={Jacobians isomorphic to a product of two elliptic curves and ternary quadratic forms},
        date={2014},
     journal={J. Number Theory},
      volume={139},
       pages={138\ndash 174},
         url={https://doi.org/10.1016/j.jnt.2013.12.006},
      review={\MR{3173190}},
}

\bib{MJ}{article}{
      author={Kani, E.},
       title={The moduli spaces of {J}acobians isomorphic to a product of two elliptic curves},
        date={2016},
     journal={Collect. Math.},
      volume={67},
       pages={21\ndash 54},
         url={https://doi.org/10.1007/s13348-015-0148-9},
      review={\MR{3439838}},
}

\bib{ESCII}{article}{
      author={Kani, E.},
       title={Elliptic subcovers of a curve of genus 2 {II}. {T}he refined {H}umbert invariant},
        date={2018},
     journal={J. Number Theory},
      volume={193},
       pages={302\ndash 335},
         url={https://doi.org/10.1016/j.jnt.2018.05.011},
      review={\MR{3846811}},
}

\bib{ESCI}{article}{
      author={Kani, E.},
       title={Elliptic subcovers of a curve of genus 2. {I}. {T}he isogeny defect},
        date={2019},
     journal={Ann. Math. Qu\'{e}.},
      volume={43},
      number={2},
       pages={281\ndash 303},
         url={https://doi.org/10.1007/s40316-018-0105-6},
      review={\MR{3996071}},
}

\bib{SubcoversofCurves}{incollection}{
      author={Kani, E.},
       title={Subcovers of curves and moduli spaces},
        date={(2021)},
   booktitle={Geometry at the frontier},
      series={Contemp. Math.},
      volume={766},
   publisher={Amer. Math. Soc., RI},
       pages={229\ndash 250},
         url={https://doi.org/10.1090/conm/766/15384},
      review={\MR{4248056}},
}

\bib{refhum}{article}{
      author={Kani, E.},
       title={The refined {Humbert} invariant for abelian product surfaces with complex multiplication},
        date={Preprint, 23 pages},
}

\bib{generalizedhumbert}{article}{
      author={Kani, E.},
       title={Generalized {Humbert} schemes and intersections of {Humbert} surfaces},
        date={Preprint, 35 pages},
}

\bib{cas}{article}{
      author={Kani, E.},
       title={Curves of genus 2 on abelian surfaces},
        date={Preprint, 37 pages},
}

\bib{Kir}{article}{
      author={Kir, H.},
       title={The classification of the refined {H}umbert invariant},
        date={Preprint, 36 pages},
}

\bib{lmfdb}{misc}{
      author={{LMFDB Collaboration}, The},
       title={The {L}-functions and modular forms database},
         how={\url{https://www.lmfdb.org}},
        date={2023},
        note={[Online; accessed 21 October 2023]},
}

\bib{Mcmullen}{article}{
      author={McMullen, C.},
       title={Teichm\"{u}ller curves in genus two: discriminant and spin},
        date={2005},
     journal={Math. Ann.},
      volume={333},
      number={1},
       pages={87\ndash 130},
         url={https://doi.org/10.1007/s00208-005-0666-y},
      review={\MR{2169830}},
}

\bib{damien2}{article}{
      author={Milio, E.},
      author={Robert, D.},
       title={Modular polynomials on {H}ilbert surfaces, https://hal.science/hal-01520262v2},
        date={2017},
         url={https://hal.science/hal-01520262v2},
}

\bib{milne1986jacobian}{incollection}{
      author={Milne, J.~S.},
       title={Jacobian varieties},
        date={1986},
   booktitle={Arithmetic geometry ({S}torrs, {C}onn., 1984)},
   publisher={Springer, New York},
       pages={167\ndash 212},
      review={\MR{861976}},
}

\bib{orbifolds}{article}{
      author={Mukamel, R.~E.},
       title={Orbifold points on {T}eichm\"{u}ller curves and {J}acobians with complex multiplication},
        date={2014},
        ISSN={1465-3060,1364-0380},
     journal={Geom. Topol.},
      volume={18},
      number={2},
       pages={779\ndash 829},
         url={https://doi.org/10.2140/gt.2014.18.779},
      review={\MR{3180485}},
}

\bib{mumford1970abelian}{book}{
      author={Mumford, D.},
       title={Abelian varieties},
   publisher={2nd edn. Oxford University Press, London},
        date={(1970)},
      review={\MR{282985}},
}

\bib{damien}{article}{
      author={Robert, D.},
       title={Eﬀicient algorithms for abelian varieties and their moduli spaces.},
        date={2021},
     journal={Université de Bordeaux (UB)},
         url={https://hal.science/tel-03498268/document},
}

\bib{Shaska_subcovers}{incollection}{
      author={Shaska, T.},
       title={Genus 2 curves with {$(3,3)$}-split {J}acobian and large automorphism group},
        date={2002},
   booktitle={Algorithmic number theory ({S}ydney, 2002)},
      series={Lecture Notes in Comput. Sci.},
      volume={2369},
   publisher={Springer, Berlin},
       pages={205\ndash 218},
         url={https://doi.org/10.1007/3-540-45455-1_17},
      review={\MR{2041085}},
}

\bib{shaska2004elliptic}{incollection}{
      author={Shaska, T.},
      author={V\"{o}lklein, H.},
       title={Elliptic subfields and automorphisms of genus 2 function fields},
        date={2004},
   booktitle={Algebra, arithmetic and geometry with applications ({W}est {L}afayette, {IN}, 2000)},
   publisher={Springer, Berlin},
       pages={703\ndash 723},
      review={\MR{2037120}},
}

\bib{MitaniShioda}{incollection}{
      author={Shioda, T.},
      author={Mitani, N.},
       title={Singular abelian surfaces and binary quadratic forms},
        date={(1974)},
   booktitle={Classification of algebraic varieties and compact complex manifolds},
   publisher={Springer, Berlin},
       pages={259\ndash 287. Lecture Notes in Math., Vol. 412},
      review={\MR{0382289}},
}

\bib{silverman1994advanced}{book}{
      author={Silverman, J.~H.},
       title={Advanced topics in the arithmetic of elliptic curves},
      series={Graduate Texts in Mathematics},
   publisher={Springer-Verlag, New York},
        date={1994},
      volume={151},
        ISBN={0-387-94328-5},
         url={https://doi.org/10.1007/978-1-4612-0851-8},
      review={\MR{1312368}},
}

\bib{smith}{book}{
      author={Smith, J. H.~S.},
       title={On the orders and genera of ternary quadratic forms \emph{{(1867)}}},
   publisher={In: Collect. Math. Papers vol. I, Oxford, pp. 455\textsc{\textendash}509},
        date={(1894)},
}

\bib{van2012hilbert}{book}{
      author={van~der Geer, G.},
       title={Hilbert modular surfaces},
      series={\emph{Ergebnisse der Mathematik und ihrer Grenzgebiete (3)}},
   publisher={Springer-Verlag, Berlin},
        date={(1988)},
      volume={16},
         url={https://doi.org/10.1007/978-3-642-61553-5},
      review={\MR{930101}},
}

\bib{watson1960integral}{book}{
      author={Watson, G.~L.},
       title={Integral quadratic forms},
   publisher={Cambridge U. Press, Cambridge},
        date={(1960)},
      review={\MR{118704}},
}

\end{biblist}
\end{bibdiv}

\end{document}